\documentclass[11pt]{amsart}

\usepackage[all]{xy}
\usepackage{mathrsfs}
\usepackage{amsmath,amsthm, amssymb, amscd, amsgen,amsxtra,amsfonts, amsbsy, stmaryrd}
\usepackage[all]{xy}
\usepackage{longtable}
\usepackage{verbatim}

\usepackage{color}
\definecolor{shadecolor}{gray}{0.875}
\definecolor{dblue}{rgb}{0,0,.6}
\usepackage[colorlinks=true, linkcolor=dblue, citecolor=dblue, filecolor = dblue, menucolor = dblue, urlcolor = dblue]{hyperref}

\usepackage{etex}
\usepackage{calligra,mathrsfs}
\usepackage{mathtools}
\usepackage{extarrows}
\usepackage{ifthen}

\usepackage{graphicx}
\usepackage{caption,rotating}
\usepackage{color}
\usepackage{subfigure}
\usepackage{wasysym}

\usepackage{times}
\newcommand{\mathds}[1]{{\mathbb #1}}

\theoremstyle{definition}
\newtheorem{Definition}{Definition}[section]

\newtheorem{Remark}[Definition]{Remark}

\theoremstyle{plain}
\newtheorem{Theorem}[Definition]{Theorem}

\newtheorem{Proposition}[Definition]{Proposition}

\newtheorem{Lemma}[Definition]{Lemma}
\newtheorem{Corollary}[Definition]{Corollary}
\newtheorem*{Corollaryx}{Corollary}

\newtheoremstyle{voiditstyle}{3pt}{3pt}{\itshape}{\parindent}%
{\bfseries}{.}{ }{\thmnote{#3}}%
\theoremstyle{voiditstyle}

\newtheoremstyle{voidromstyle}{3pt}{3pt}{\rm}{\parindent}%
{\bfseries}{.}{ }{\thmnote{#3}}%
\theoremstyle{voidromstyle}

\newcommand{\cal}{\mathcal}
\newcommand{\Aff}{{\mathds{A}}}

\newcommand{\CC}{{\mathds{C}}}

\newcommand{\FF}{{\mathds{F}}}
\newcommand{\FFcl}{{\overline{\mathds{F}}}}
\newcommand{\GG}{{\mathds{G}}}
\newcommand{\HH}{{\mathds{H}}}

\newcommand{\ZZ}{{\mathds{Z}}}
\newcommand{\PP}{{\mathds{P}}}
\newcommand{\QQ}{{\mathds{Q}}}

\newcommand{\OO}{{\cal O}}

\newcommand{\Spec}{{\rm Spec}\:}

\newcommand{\Pic}{{\rm Pic }}

\newcommand{\Jac}{{\rm Jac}}
\newcommand{\Aut}{{\rm Aut}}

\newcommand{\EXT}{\mathop{\mathscr{E}{\kern -2pt {xt}}}\nolimits}
\newcommand{\Hom}{\mathop{\mathrm{Hom}}\nolimits}
\newcommand{\HOM}{\mathop{\mathscr{H}{\kern -3pt {om}}}\nolimits}

\newcommand{\calC}{\mathscr{C}}

\newcommand{\et}{{\rm \acute{e}t}}

\renewcommand{\HH}{{\rm{H}}}
\newcommand{\Frob}{\mathrm{F}}
\newcommand{\Qell}{\mathbb{Q}_{\ell}}
\newcommand{\DPgm}{\mathcal{D}}
\newcommand{\DP}{\mathcal{D}}
\newcommand{\genpos}[1]{\mathcal{P}_{#1}}
\newcommand{\PGL}{\mathrm{PGL}}

\title{Cohomology of moduli spaces of Del Pezzo surfaces}
\author{Olof Bergvall}%
\address{Department of Electronics, Mathematics and Natural Sciences, University of G\"avle, 801 76 G\"avle, Sweden}
\email{olof.bergvall@hig.se}

\author{Frank Gounelas}%
\address{Georg-August-Universit\"at G\"ottingen, Fakult\"at f\"ur Mathematik und Informatik, Bunsenstr. 3-5, 37073 G\"ottingen, Germany}
\email{gounelas@mathematik.uni-goettingen.de}

\date{\today}

\begin{document}
\begin{abstract}
We compute the rational Betti cohomology groups of the coarse moduli spaces of geometrically marked Del Pezzo surfaces
of degree three and four as representations of the Weyl groups of the corresponding root systems.  The proof uses a
blend of methods from point counting over finite fields and techniques from arrangement complements.
\end{abstract}
\maketitle

\section{Introduction}

The geometry of the 27 lines on a smooth complex cubic surface is one of the oldest 
and most studied problems in algebraic geometry. 
For a fixed cubic surface $X\subset\PP^3$, the data of 6 disjoint lines along with the class of a
hyperplane determine what is called a geometric marking, 
i.e., a special type of basis of the Picard group $\Pic X \cong\ZZ^7$ 
(see Section \ref{section:geom background} for details). 
The isomorphism classes of such pairs of a cubic
surface along with a geometric marking form a 4-dimensional coarse moduli space $\DP_3$, 
and forgetting the marking gives a morphism 
\[
\DP_3\to\calC_3 := |\OO_{\PP^3}(3)|_{\mathrm{sm}}\sslash\PGL_4
\]
to the moduli space of cubic surfaces. 
Since the Weyl group of the lattice $E_6$ is the automorphism group of the set of lines, 
this morphism is in fact a quotient by $W(E_6)$. 
The action of the group $W(E_6)$ enriches the study of cubic surfaces 
via the study of root systems and the representation theory of Coxeter groups.

The dimensions of the rational cohomology groups $\HH^i(\DP_3,\QQ)$ have been known for some time
(at least they are easily computed via an observation of Arnold, 
see Lemma~\ref{poincarepoly} for details).
On the other hand, very little is known about the problem of describing the cohomology groups
$\HH^i(\DP_3)$ as representations of $W(E_6)$.
As the order of $W(E_6)$ is 51840 and the number of irreducible representations is 25, 
determining the decomposition of $\HH^i(\DP_3)$ into irreducible representations 
is a highly non-trivial computational task.

\subsection{The main result}

The aim of this paper is to completely determine the action of $W(E_6)$ 
on the cohomology groups $\HH^i(\DP_3,\QQ)$.  In the statement that follows, 
we use the symbol $\phi^e_d$ to denote an irreducible representation of $W(E_6)$ of
dimension $d$ such that $e$ is the smallest symmetric power of the standard representation
of $W(E_6)$ of which $\phi^e_d$ is a direct summand. The numbers $(d,e)$ uniquely determine
the 25 irreducible representations of $W(E_6)$ (see \cite[p411]{carter} for definitions and \cite[p428-429]{carter} for further details).

\begin{Theorem}\label{theo:cohdp3}
The cohomology $\HH^i(\DP_3,\QQ)$ of the moduli space of complex, geometrically marked degree 3 Del Pezzo surfaces as a
representation of the Weyl group $W(E_6)$ is as follows
\begin{equation*}
\begin{array}{ll}
        \HH^0 =& \phi_1^0 \\
        \HH^1 =& \phi_{15}^4 \\
        \HH^2 =& \phi_{81}^6 \\
        \HH^3 =& \phi_{15}^5+\phi_{80}^7+\phi_{90}^8 \\
        \HH^4 =& \phi_{10}^9 + \phi_{80}^7 + \phi_{30}^3 + \phi_{30}^{15}.
\end{array}
\end{equation*}
\end{Theorem}

To approach this problem we first employ two well-known techniques.
Firstly, in Section~\ref{sectionpointcounts}, we use point counts and Lefschetz formulas in positive
characteristic to study the action of a large subgroup of $W(E_6)$.
Secondly, in Section~\ref{section:arrangements}, we use the relation first studied by Looijenga) 
between the moduli spaces of rational surfaces with anticanonical sections and complements
of arrangements of hyperplanes and hypertori.
This allows us to make progress on the problem by computing the cohomology groups of these
arrangement complements, something which can be achieved using a 
Theorem of Macmeikan \ref{macmeikan} and significant computer algebra 
computations using algorithms developed by the first author in \cite{bergvalltor}.
To finish, we sieve out the cohomology groups $\HH^i(\DP_3,\QQ)$ by embedding them
into the cohomology of various other moduli spaces with $W(E_6)$-action.
Since the proof is somewhat technical, we provide a more detailed sketch in Section~\ref{proofsketch}.

\subsection{Some applications}
\label{applicationssubsec}
Before giving further details about the method of proof, we give some corollaries which 
exhibit the strength of the decomposition given in Theorem~\ref{theo:cohdp3}. 
First, by taking the invariant part, we see that the moduli space of unmarked Del Pezzo surfaces
has the rational cohomology of a point, recovering  the following theorem of Vasiliev. 
In combination with a Leray--Hirsch-type result of \cite{peterssteenbrink}, 
this also computes $\HH^i(\calC_3,\QQ)$.

\begin{Corollaryx}(Vasiliev, \cite[Theorem 4]{vasiliev})
The Poincar\'e polynomial of $C_3=|\OO_{\PP^3}(3)|_{\mathrm{sm}}$ is 
\[
P_{C_3}(t) := \sum_i \left(\dim \HH^i(C_3,\QQ)\right)t^i = (t^3+1)(t^5+1)(t^7+1).
\]
Since also $P_{\rm{PGL}_4(\CC)}(t)=(t^3+1)(t^5+1)(t^7+1)$, this gives that $\HH^i(\calC_3,\QQ)=0$ for $i>0$.
\end{Corollaryx}

By instead taking invariants with respect to subgroups of $W(E_6)$, 
the above representation structure can be used to compute
cohomology of various geometric quotients of $\DP_3$.  
We give two examples below but we stress that the computations
are just as simple for all of the conjugacy classes of subgroups of $W(E_6)$ 
(since there are 350 conjugacy classes of subgroups
of $W(E_6)$, see \cite{connorleemans}, we have chosen not to include an exhaustive list).

\begin{Corollaryx}(Das, \cite{dasline})
    The cohomology $\HH^i(\calC_3^L,\QQ)$ of the moduli space $\calC_3^L$ of cubic surfaces with a marked line
    are trivial for $i>0$ and $\dim \HH^0(\calC_3^L,\QQ)=1$.
\end{Corollaryx}
\begin{proof}
    Let $L$ be a line in a cubic surface $X$. The stabilizer $G$ of $L$ in $W(E_6)$ is isomorphic to the Weyl group
    $W(D_5)$ and we have that $\calC_3^L=\DP_3/G$. Thus, we obtain the cohomology of $\calC_3^L$ by taking
    $W(D_5)$-invariants in Theorem~\ref{theo:cohdp3}. This is a simple exercise in character theory. We first
    explicitly construct a copy of $W(D_5)$ in $W(E_6)$ and induce the trivial character of $W(D_5)$ to $W(E_6)$. The result is the character
    $\chi=\phi_1^0+\phi_6^1+\phi_{20}^2$. We now obtain the dimension of $\HH^i(\calC_3^L,\QQ)$ by taking the inner product
    (in the sense of character theory) of $\chi$ and the character of $\HH^i(\DP_3,\QQ)$.
\end{proof}

To explain the relevance of the next result we need some terminology.
Following \cite{dolgachev}, a sixer on a cubic surface $X$ 
is a set of six exceptional curves which are mutually orthogonal with respect
to the intersection pairing. Sixers naturally come in pairs; there 
is a unique root $\alpha$ which pairs to $1$ with all elements of the sixer
and taking the reflection given by $\alpha$ yields another sixer 
(whose corresponding root is $-\alpha$). Such a pair of sixers is called
a double-six.

\begin{Corollaryx}
    The cohomology $\HH^i(\calC_3^D,\QQ)$ of the moduli space of cubic surfaces with a marked
    double-six is trivial for $i>1$ and of dimension 1 for $i=0,1$.
\end{Corollaryx}

\begin{proof}
    The proof follows the same path as above. 
    Let $G$ be the stabilizer of a double-six (a degree two extension
    of $S_6$).  Then $\calC_3^D=\DP_3/G$ and a direct computation gives that the character of $G$ is
    $\chi=\phi_{1}^0+\phi_{15}^{4}+\phi_{20}^{2}$ and we conclude as before.
\end{proof}

In another direction, one can consider the following two questions over finite fields.
First, it is well known that for a cubic surface $X$ over $\FF_q$, we have 
\[|X(\FF_q)| = q^2+a(X)q+1\]
where $a(X)$ is the trace of Frobenius acting on the Picard group of $X$. 
In \cite{serre}, Serre asked which values of $a(X)$ can occur. 
This question was recently answered in \cite{banwaitetal}. 
Theorem~\ref{theo:cohdp3} encodes a refined answer to this question, 
namely the precise number of cubic surfaces over $\FF_q$ with trace $a$. 
Secondly, one can consider the inverse Galois problem of cubic surfaces over $\FF_q$, i.e., 
the problem of finding which conjugacy classes of $W(E_6)$ can arise as the image of the 
Frobenius action on the Picard group of a cubic surface $X$ over $\FF_q$. 
Also this question was settled, in \cite{loughrantrepalin}, 
just before the appearance of the present paper.
Theorem~\ref{theo:cohdp3} encodes a refined answer to this question as well, 
namely it gives the precise number of cubic surfaces over $\FF_q$ 
for which Frobenius acts as a given conjugacy class in $W(E_6)$.  
Both of the results are obtained from  Theorem~\ref{theo:cohdp3} 
via a straightforward application of Lefschetz's trace formula 
(for Serre's question one also has to add up conjugacy classes with the same trace).

\subsection{Betti numbers}
As mentioned above, the dimensions of the cohomology groups $\HH^i(\DP_3,\QQ)$ are
much easier to compute than the full proof of Theorem~\ref{theo:cohdp3}.
At the same time the result seems to be missing from the literature.
We therefore sketch the computation of the Poincar\'e polynomial of $\DP_3$ using
a fibration usually attributed to Arnold (although we have not been able to find a reference) and
the Orlik--Solomon formula as well as via point counts in positive characteristic.

\begin{Lemma}\label{poincarepoly}
The Poincar\'e polynomial of $\DP_3$ is as follows
    $$P_{\DP_3}(t) = 1+15t+81t^2+185t^3+150t^4.$$
\end{Lemma}

\begin{proof}
The variety $\DP_3$ is isomorphic to the moduli space $\genpos{6}$ 
of 6 points in the plane in general position up to projective equivalence, see \cite{dolgachevortland}.
The forgetful map $\genpos{6} \to \genpos{5}$ is a fibration 
with fibre $F$ a $\PP^2$ with the 10 lines through pairs of the first five points 
as well as the conic through the first five points taken away. 
By the Orlik--Solomon formula, the fibre $F$ has Poincar\'e polynomial $1+10t+25t^2$. 
By taking the first four points to the points $[1:0:0]$, $[0:1:0]$, $[0:0:1]$ and $[1:1:1]$ 
we identify $\genpos{5}$ with the projective plane with the 6 lines through pairs of these points taken away. Again, using the Orlik--Solomon formula we obtain that $\genpos{5}$ has Poincar\'e polynomial $1+5t+6t^2$. The result now follows from the K\"unneth decomposition.

Alternatively (see Section \ref{sixpoinsec} for details on what follows), 
it is elementary to check that $|\mathcal{P}_{6,\overline{\FF}_q}^\Frob|$, 
the number of Frobenius-fixed points of the same variety considered over $\overline{\FF}_q$, 
is given by the formula 
$$q^4 - 15q^3 + 81q^2 - 185q + 150.$$ 
The Lefschetz trace formula gives that 
$$|\mathcal{P}_{6,\overline{\FF}_q}^\Frob|=\sum_{k\geq0}(-1)^k\rm{Tr}\left(\Frob,
    \HH^k_{\et,c}(\mathcal{P}_{6,\overline{\FF}_q}, \Qell) \right).$$ 
A result of Dimca--Lehrer now implies that Frobenius acts by multiplication with $q^{i-4}$ 
on the $i$-th \'etale cohomology group with compact support, so one concludes using Poincar\'e duality.
\end{proof}

\subsection{Sketch of the proof}
\label{proofsketch}
The proof of Theorem \ref{theo:cohdp3}, now proceeds as follows. 
First, as noted in the proof of Lemma~\ref{poincarepoly}, we have an isomorphism
$\DP_3\cong\genpos{6}$ which is equivariant with respect to the $S_6$-action. 
Results of Dimca--Lehrer \cite{dimcalehrer} imply that Frobenius acts
with all eigenvalues equal to $q^{i-4}$ on $\HH^i_{\et,c}(\mathcal{P}_{6,\overline{\FF}_q},\QQ_\ell)$.
This allows us, using Lefschetz trace formulas, equivariant point counts and Poincar\'e duality, 
to determine $\HH^i_{\et}(\mathcal{P}_{6,\overline{\FF}_q},\QQ_\ell)\cong \HH^i(\mathcal{P}_{6,\CC},\CC)$
as a representation of $S_6$. The group $S_6$ is rather small in comparison to $W(E_6)$ however, so this does not give the complete picture. 

To be more precise, there are still a number of different irreducible representations of $W(E_6)$ that restrict to the
same ones of $S_6$, so another idea is required. We follow Looijenga \cite{looijenga} by constructing covers of $\DP_3$
parametrising further structure, namely singular anticanonical sections along with a singular point. These are higher
dimensional varieties yet are complements of hyperplane or toric arrangements. As such, 
their cohomology is richer, yet easier to compute.

To give an explicit example (see Sections \ref{section:anticanonicalsections}, \ref{section:arrangements} for further
particulars), let $\DP_3^{\rm n}$ be the coarse moduli space parametrising triples $(X,A,p)$ where $A\in|-K_X|$
and $p\in A$ is the only node on $A$. The forgetful morphism $\DP_3^{\rm n}\to\DP_3$ has 2-dimensional fibres, and we
prove that there is a $W(E_6)$-equivariant isomorphism $$\DP_3^{\rm n} \cong (\ZZ/2\ZZ) \setminus \left(
(\CC^*)^6-\bigcup_{i=1}^{36}(\CC^*)^5 \right)$$ to the complement of 36 hypertori in a 6-dimensional torus modulo a
natural action by the group of two elements. This allows the computation of $\HH^i(\DP_3^{\rm n})$ as a
$W(E_6)$-representation by an Orlik--Solomon-type formula of Macmeikan and an algorithm of the first author
\cite{bergvalltor}.

We stress that there is no hope of performing these computations by hand. For example the Poincar\'e polynomial of $\DP_3^n$ is
\[
1 + 36t + 525t^2 + 3960t^3 + 16299t^4 + 34884t^5 + 30695t^6
\]
and the size of the posets being summed over in Theorem~\ref{macmeikan} are often 
in the tens of thousands, so the computer algebra package Sage is used here
\footnote{The code can be found at the repository
\begin{tiny}\url{https://github.com/OlofBergvall/CohTorArr}\end{tiny} (most of it can also been found in the first
author's doctoral thesis).}. See also \cite{BergvallGD} and \cite{BergvallEJM}  as well as
\cite{renetal} where similarly large numbers arise in the study of moduli spaces of tropical Del Pezzo surfaces of
degree 3 and 4.

Going back to the proof of Theorem \ref{theo:cohdp3}, consider now also the space $\DP_3^{\rm c}$ of cuspidal points,
which has a similar description in terms of hyperplane arrangements, and let $U=\DP_3^{\rm n}\sqcup\DP_3^{\rm
c}\to\DP_3$. We have that $U$ is an open subset of the moduli space $\DP_3^{\rm a}$ parametrising triples $(X,A,p)$
where $A\in|-K_X|$ is singular and $p\in A$ is a singularity of $A$. We observe that the forgetful morphism from
$\DP_3^{\rm a}$ to $\DP_3$ is proper so we obtain an injection of mixed Hodge structures
$$\HH^i(\DP_3,\CC)\hookrightarrow\HH^i(\DP_3^{\rm a},\CC).$$
By analysing the mixed Hodge structures of $\DP_3^{\rm a}$, $U$ and $Z=\DP_3^{\rm a}\setminus U$ one sees that we in
fact get an injection of mixed Hodge structures
\[ \HH^i(\DP_3,\CC)\hookrightarrow\HH^i(U,\CC).  \]
The groups $\HH^i(U,\CC)$ are easily computed from the cohomologies of $\DP_3^{\rm n}$ and $\DP_3^{\rm c}$. Performing
similar operations for various such spaces $U$ allows us to sieve out the remaining extraneous irreducible
representations and conclude the result.

\subsection{Further comments}
The study of moduli spaces of rational surfaces with anticanonical cycles, such as those appearing in the proof of
Theorem~\ref{theo:cohdp3}, goes back to Looijenga's paper \cite{looijenga81}.  The strata of these moduli spaces are
arguably interesting in their own right, in particular in light of the close analogy between these strata and
Kontsevich-Zorich strata of moduli spaces of abelian differentials, see \cite{kontsevichzorich}.  The moduli theory of
rational surfaces with anticanonical cycles has experienced a tremendous resurgence over the past two decades, in
particular through their relation to cluster algebras and mirror symmetry of log-Calabi-Yau varieties, see for example
the series of papers by Gross, Hacking and Keel \cite{grosshackingkeel15b, grosshackingkeel15c, grosshackingkeel15a}.
Our results thus represent an addition to this exciting field.

The moduli space of geometrically marked Del Pezzo surfaces of degree 4 is isomorphic to the moduli space $\mathcal{M}_{0,5}$
of smooth rational curves with 5 marked points (the isomorphism is given by representing a geometrically marked Del Pezzo surface of degree $4$ 
by 5 points in $\PP^2$ in general position and sending this $5$-tuple to the isomorphism class of the conic through the points). 
The cohomology of $\mathcal{M}_{0,n}$ is well understood.
Nevertheless, it is worthwhile to observe that our method works also in this case; for one, it gives new cohomological
information about strata of moduli spaces of rational surfaces but it also adds a new perspective to one of
the most well-studied areas of moduli theory. Our result is the following.


\begin{Theorem}(See Section \ref{fivepointsec})\label{theo:cohdp4}
The cohomology $\HH^i(\DP_4,\QQ)$ of the moduli space of complex, geometrically marked degree 4 Del Pezzo surfaces as a
representation of the Weyl group $W(D_5)$ is as follows
\begin{equation*}
    \HH^0=\phi^0_1,\hspace{5pt} \HH^1=\phi^4_5, \hspace{5pt} \HH^2=\phi^6_6.
\end{equation*}
\end{Theorem}

After completing this paper, we were informed that at the same time as us, Das--O'Connor \cite{doc} have also computed
the equivariant point counts of $\mathcal{P}_5$ of Table \ref{S5S6table}, and hence also the cohomology of $\DP_4$ as a
$S_5$-representation. It should also be mentioned that Glynn \cite{glynn} has counted the points of $\mathcal{P}_5$
non-equivariantly. See also \cite{kklpw} for further results in the direction of coding theory and counting points in
general linear position.

\subsection*{Acknowledgements}
We would like to thank Hsueh-Yung Lin for the reference \cite{wells} and Roberto Laface for clarifications on elliptic
fibrations. We would especially like to thank Ronno Das for spotting an error which led to wrong results in a previous
version of Theorem \ref{theo:cohdp3}. The second author acknowledges the support of the ERC Consolidator Grant 681838
K3CRYSTAL.


\section{Geometric background}\label{section:geom background}


We begin with some standard definitions and properties of Del Pezzo surfaces that can be found for example in \cite[\S 8]{dolgachev}. A Del Pezzo surface of degree $d$ is a smooth complex proper surface $X$ so that the anticanonical divisor $-K_X$ is ample and satisfies $K_X^2=d$. One proves that $1\leq d\leq9$. A Del Pezzo surfaces of degree $d$ is isomorphic to the blowup of $\PP^2$ in $9-d$ points except if $d=8$ where $X$ is either the blow up of $\mathbb{P}^2$ in a single point or
$X \cong \PP^1 \times \PP^1$.

The Picard rank of such a surface is $\rho=10-d$. Given $r=9-d$ points $P_1, \ldots, P_r\in\PP^2$ in general position, the Del Pezzo surface
$\pi:X\to\PP^2$ of degree $d$ obtained by blowing up these points has basis for the Picard group
$$
\mathrm{Pic}(X) = \ZZ L \oplus \ZZ E_1 \oplus \cdots \oplus \ZZ E_r,
$$
where $L=\pi^*\OO(1)$ is the strict transform of a line in $\PP^2$ and $E_i$ is the exceptional curve
which is the inverse image of $P_i$. Such a basis coming from a blowup is called a \emph{geometric marking}. Note that $-K_X.E_i=1$ so the exceptional divisors are lines under the anticanonical map $X\to\PP^d$ (which is a closed embedding if $d\geq3$). Of course, there are usually other such lines contained in $X$, leading to the fact that $X$ can be represented as the blowup of projective space in multiple ways, and also that there can be many geometric markings. We describe now the group of automorphisms of such markings.

For $S \subset \{1,\ldots,r\}$ let
 \begin{align*}
  \gamma_S & = \sum_{i \in S} E_i, \\
  \alpha_{ij} & = E_i - E_j, \quad i < j, \\
  \alpha_S & = L - \gamma_S, \quad |S|=3, \\
  \alpha_S & = 2L - \gamma_S, \quad |S|=6, \\
  \alpha_i & = 3L - \gamma_{\{1, \ldots, 8\}} -E_i.
 \end{align*}
The set $\Phi_d$ consisting of the above elements and their negatives is then a root system
of type
\begin{equation*}
 \begin{array}{|c|cccc|}
 \hline
 d & 1 & 2 & 3 & 4 \\
 \hline
 \Phi_d & E_8 & E_7 & E_6 & D_5 \\
 \hline
 \end{array}
\end{equation*}
spanning the orthogonal complement $K_X^{\perp}$ (with respect to the intersection pairing) of the canonical class.
We thus have that the Weyl group $W(\Phi_d)$ of the root system $\Phi_d$ acts
on the set of geometric markings of a Del Pezzo surface of degree $d$. As is usual, we will denote by $\ell, k, e_1,\ldots,e_r$ the elements of the corresponding lattice $L_{\Phi_d}$, sent respectively to $L,K_X, E_1,\ldots, E_r$.

The coarse moduli space of Del Pezzo surfaces is of dimension $\max\{0, 2(5-d)\}$. Expressed as the blowup of $\PP^2$, a Del Pezzo surface comes naturally equipped with a geometric marking and this construction gives rise to a moduli space $\DPgm_d$ of Del Pezzo surfaces along with a geometric marking. The map forgetting the marking is finite and is the quotient by the Weyl group $W(\Phi_d)$.

\section{Point counts}\label{sectionpointcounts}

Let $p$ be a prime number, $n$ a positive integer and let $q=p^n$.
Let $\FF_q$ denote a finite field with $q$ elements,
$\FF_{q^m}$ a degree $m$ extension and let $\FFcl_{q}$ denote an algebraic
closure of $\FF_q$. Let $X$ be a scheme separated and of finite type over $\overline{\FF}_q$ which is defined over $\FF_q$, and let
$\Frob=\Frob_{X/\overline{\FF}_q}:X\to X$ denote the relative (i.e.\ linear) $q$-power Frobenius endomorphism.
Let $\ell \neq p$ be another prime number and let $\rm{H}_{\et,c}^i(X,\Qell)$ denote the $i$-th compactly supported $\ell$-adic cohomology
group of $X$, noting that the induced action of $F$ on this group is geometric Frobenius. Recall that if $X$ is smooth
and integral (but not necessarily proper) of dimension $n$, we have Poincar\'e duality which in this case gives
$\HH^i_{\et,c}(X,\Qell)\cong \HH^{2n-i}_{\et}(X,\Qell)$. Let $\Gamma$ be a finite group of automorphisms of $X$ defined
over the base field $\FF_q$.
For an element $\sigma \in \Gamma$, we write $\Frob \sigma$ for the composition
and $\left|X^{\Frob\sigma} \right|$ for the number of fixed points of $\Frob\sigma$ in $X(\overline{\FF}_q)$.

\begin{Definition}
 The determination of $\left|X^{\Frob \sigma} \right|$ for all $\sigma \in \Gamma$ is called a $\Gamma$-\emph{equivariant point count} of $X$ over $\FF_q$.
\end{Definition}

Recall the following form of the \emph{Lefschetz fixed point formula} (e.g.\ \cite[\S 3]{delignelusztig})
$$
|X^{F\sigma}|=\sum_{k\geq0}(-1)^k\rm{Tr}\left(\Frob\sigma, \HH^k_{\et,c}(X, \Qell) \right).
$$
We will now show how point counts can lead to information about $\HH^k_{\et,c}(X, \Qell)$ as a $\Gamma$-representation, in the following particular case.
\begin{Definition}\label{defn:minpure}
Let $X$ as above also be integral. We say that $X$ is \emph{minimally pure} if $F$ acts on $\HH^i_{\et,c}(X,\Qell)$ with eigenvalues equal to $q^{i-\dim X}$.
\end{Definition}
For example, for $X$ minimally pure, the coefficient of a term of the form $q^{k-\dim X}$ appearing in a computation of $|X^{F\sigma}|$ must necessarily be a contribution from $\HH^k_{\et,c}(X,\Qell)$ in the following sense (cf.\ \cite[2.6]{kisinlehrer})
\begin{eqnarray*}
\left|X^{\Frob\sigma}\right| &=&  \sum_{k\geq0}(-1)^k\rm{Tr}\left(\Frob\sigma, \HH^k_{\et,c}(X, \QQ_\ell) \right) \\
&=& \sum_{k\geq0}(-1)^k\rm{Tr}\left(\sigma, \HH^k_{\et,c}(X, \QQ_\ell) \right)q^{k-\dim X},
\end{eqnarray*}
where the second equality is a consequence of minimal purity and the fact that $\Frob\sigma$ is $\Frob$ on the $\sigma$-twist of $X$.
Note now that the value of the character of the representation is determined by a single representative of a conjugacy class, so in order to completely determine the $\Gamma$-representations $\HH^k_{\et,c}(X,\Qell)$, it will suffice to perform the equivariant point counts for representatives of conjugacy classes of $\Gamma$.

Recall that the moduli space $\DPgm_d$ of geometrically marked Del Pezzo surfaces of degree $d$ is isomorphic to the moduli space $\genpos{n}$ of $n=9-d$ points in general position in the projective plane up to projective equivalence. The symmetric group
$S_n$ acts as automorphisms on $\genpos{n}$ by permuting the points. In the following two sections we will count the fixed points of $\genpos{5}$ and $\genpos{6}$ equivariantly with respect to $S_5$ and $S_6$ over a general finite field $\FF_q$.

\subsection{Five points in the projective plane}
\label{fivepointsec}
The computations are not very complicated, and since rather similar and somewhat tedious, we only present details in the two extremal cases (the identity element and a $5$-cycle) to illustrate some aspects of what is going on.

\begin{Proposition}
\label{d4idprop}
 The number of fixed points of $\Frob$ in $\genpos{5}$ is
 \begin{equation*}
     \left| \genpos{5}^{\Frob} \right| = q^2-5q+6.
 \end{equation*}
\end{Proposition}

\begin{proof}
 In order for a quintuple of points to be fixed by Frobenius they
must all be defined over $\FF_q$. Since the quintuple is in general
position, there is an element in $\PGL(3,\FF_q)$ taking the first four
points to the points $P_1 = [1:0:0]$, $P_2=[0:1:0]$, $P_3=[0:0:1]$ resp. $P_4=[1:1:1]$.
Thus, all we need to do is choose a point $P_5$ away from the six lines
$L_1, \ldots, L_6$ between pairs of
these four points. The lines $L_1, \ldots, L_6$ intersect at the points
$P_1$, $P_2$, $P_3$ and $P_4$ and at three further points. Furthermore we
have that each line contains exactly three points of intersection. Since $|\PP^2(\FF_q)|=q^2+q+1$ and $|\PP^1(\FF_q)|=q+1$
there are
\begin{equation*}
    q^2+q+1-6(q-2)-7=q^2-5q+6
\end{equation*}
possible choices for $P_5$.
\end{proof}

 We now turn to the case of a $5$-cycle.

\begin{Lemma}
\label{linelemma}
 Let $P_1, \ldots, P_5$ be five points in $\PP^2$
 permuted cyclically by $\Frob$.
 If three of them lie on a line $L$, then $L$ is defined
over $\FF_q$ and contains all five points.
\end{Lemma}

\begin{proof}
 Let $P_i$, $P_j$ and $P_k$ be the points on $L$ and let $S=\{P_i,P_j,P_k\}$.
 Then either $|S \cap \Frob S| = 2$ or $|S \cap \Frob^2S|=2$ so either
 $L=\Frob L$ or $L=\Frob^2 L$. Thus, $L$ is defined over $\FF_{q^5}$ and
 $\FF_q$ or $\FF_{q^2}$. Since $5$ is coprime to both $1$ and $2$ we can in both cases conclude that $L$ is defined over $\FF_q$ and that it thus contains
 all five points.
\end{proof}

\begin{Proposition}
 Let $\sigma \in S_5$ be a 5-cycle. Then
 \begin{equation*}
     \left| \genpos{5}^{\Frob \sigma} \right| = q^2+1.
 \end{equation*}
\end{Proposition}

\begin{proof}
 We want to count quintuples of points $P_1, \ldots, P_5\in\PP^2(\FF_q)$ which are
 \begin{itemize}
     \item[(\emph{i})] permuted cyclically by Frobenius, i.e. $P_1$ is
     defined over $\FF_{q^5}$ but not over $\FF_q$ and $P_i=\Frob^{i-1}P_1$ and
     \item[(\emph{ii})] in general position.
 \end{itemize}
 There are
 \begin{equation*}
     q^{10}+q^5+1-(q^2+q+1) = q^{10}+q^5-q^2-q
 \end{equation*}
 quintuples satisfying (\emph{i}) and, by Lemma~\ref{linelemma}, we only
 need to make sure that a quintuple does not lie on a $\FF_q$-line
 for it to be in general position. We thus pick one of the $q^2+q+1$
 lines $L$ in $\PP^2(\FF_q)$ defined over $\FF_q$ and then a point $P_1$
 on $L$ defined over $\FF_{q^5}$ but not $\FF_q$ in one of $(q^5+1)-(q+1)=q^5-q$ ways. We conclude that the number of quintuples satisfying both (\emph{i}) and
 (\emph{ii}) is
 \begin{equation*}
     q^{10}+q^5-q^2-q - (q^2+q+1)(q^5-q) = q^{10}-q^7-q^6+q^3.
 \end{equation*}
 We divide by $|\PGL(3,\FF_q)|=(q^2+q+1)(q^3-q)(q^3-q^2)$ to obtain
 $\left| \genpos{5}^{\Frob \sigma} \right| = q^2+1$.
\end{proof}

The remaining cases are similar and we give the results in Table~\ref{S5S6table}.
\begin{table}
\begin{equation*}
    \begin{array}{|l|l||l|l|}
 \hline
 &&&\\[-1em]
        {[\sigma]}\in\rm{C}(S_5) &  |\genpos{5}^{\Frob \sigma}| & {[\sigma]}\in\rm{C}(S_6) &  |\genpos{6}^{\Frob \sigma}| \\
 &&&\\[-1em]
        \hline
 &&&\\[-1em]
         {[(12345)]} & q^2+1 & {[(123456)]} & q^4 - q^3  \\
         {[(1234)]} & q^2+q & {[(12345)]} & q^4 + q^2  \\
         {[(123)(45)]} & q^2-q & {[(1234)(56)]} & q^4 - q^3 - q^2 - q - 2 \\
         {[(123)]} & q^2+q & {[(1234)]} & q^4 + q^3 - q^2 - q  \\
         {[(12)(34)]} & q^2-q-2 & {[(123)(456)]} & q^4 - 3q^3 - 2q + 12 \\
         {[(12)]} & q^2-q & {[(123)(45)]} & q^4 - 2q^3 + q  \\
         {[\mathrm{id}]} & q^2-5q+6 & {[(123)]} & q^4 + q  \\
         && {[(12)(34)(56)]} & q^4 - q^3 - 3q^2 + 3q  \\
         && {[(12)(34)]} & q^4 - 3q^3 - 3q^2 + 7q + 6  \\
         && {[(12)]} & q^4 - 5q^3 + 9q^2 - 5q  \\
         && {[\mathrm{id}]} & q^4 - 15q^3 + 81q^2 - 185q + 150 \\
 \hline
    \end{array}
\end{equation*}
\caption{The $S_5$- resp. $S_6$-equivariant point count of $\genpos{5}$ and $\genpos{6}$ over $\FF_q$ for all conjugacy classes.}
\label{S5S6table}
\end{table}

\begin{proof}(of Theorem \ref{theo:cohdp4})
As explained in the proof of Proposition~\ref{d4idprop}, the space $\genpos{5}$ is isomorphic to $\PP^2 \setminus \Delta$ where
$\Delta$ is the normal crossings union of $6$ lines.
Thus, $\genpos{5}$ is isomorphic to $\Aff^2 \setminus \mathring{\Delta}$
where $\mathring{\Delta}$ is the union of $5$ lines. By the results of
Dimca--Lehrer \cite{dimcalehrer}, such a space is minimally pure.
We thus see that $\HH^i_{\et}(\mathcal{P}_{5,\overline{\FF}_q},\Qell)$ takes the values on the conjugacy classes of $S_5$ as given in Table~\ref{d4vals}. Note now that $\genpos{5}$ is smooth over $\Spec\ZZ$ and admits $\PP^2$ as a compactification, so from \cite[Corollary 1.3]{kisinlehrer} we obtain that in such a case a base change isomorphism exists for the quasiprojective variety in question and thus $S_5$-equivariant comparison isomorphisms
$\HH^i_{\et}(\mathcal{P}_{5,\overline{\FF}_q},\QQ_\ell) \cong  \HH^i(\mathcal{P}_{5,\CC},\CC)$. In other words, the results of Table~\ref{d4vals} hold also for $\HH^i(\DP_4)$.
\begin{table}[h]\centering
    \begin{equation*}
        \begin{array}{|c|rrrrrrr|}
 \hline
 &&&&&&&\\[-1em]
     \, & [\mathrm{id}] & [(12)] & [(12)(34)] & [(123)] & [(123)(45)] & [(1234)] & [(12345)] \\
     \hline
 &&&&&&&\\[-1em]
     \HH^0 & 1 & 1 & 1 & 1 & 1 & 1 & 1  \\
     \HH^1 & 5 & 1 & 1 & -1 & 1 & -1 & 0 \\
     \HH^2 & 6 & 0 & -2 & 0 & 0 & 0 & 1 \\
 \hline
        \end{array}
    \end{equation*}
    \caption{The values of $\HH^i(\genpos{5})$ at the conjugacy classes of $S_5$.}
    \label{d4vals}
\end{table}
From this we see that $\HH^0$ is the trivial representation, that $\HH^1$ is the irreducible five dimensional representation of $S_5$ corresponding to the partition $[3,2]$ and that $\HH^2$ is the exterior square of the standard representation of $S_5$ (also an irreducible representation corresponding to the partition $[3,1^2]$). There is precisely one representation of $W(D_5)$ restricting to each of these representations (namely the irreducible representations $\phi_1^0$, $\phi_5^4$ resp. $\phi_6^6$ of \cite{carter}), which leads to a proof of Theorem \ref{theo:cohdp4}.
\end{proof}

\begin{Remark}
While the fact that the  cohomology groups are irreducible as representations of $W(D_5)$ is unexpected, it is not surprising that the cohomology groups are entirely determined by the $S_5$-equivariant point count. The action of $W(D_5)$ on $\DPgm_4$ factors through $W(D_5)/\mathrm{Aut}(S)$, where $\mathrm{Aut}(S)\cong(\ZZ/2\ZZ)^4$ is the automorphism group of a general degree $4$ Del Pezzo surface, and the quotient $W(D_5)/\mathrm{Aut}(S)$ is isomorphic to $S_5$.
\end{Remark}

\subsection{Six points in the projective plane}
\label{sixpoinsec}

The task of finding the number of fixed points of $\Frob \sigma$ in $\genpos{6}$ for each element of $S_6$ is complicated by the fact that we now also need to make sure that the points do not lie on a conic. Nevertheless, the computations are rather straightforward and we content ourselves with giving the results in Table~\ref{S5S6table}.
We will see in Corollary \ref{dp3puritycor} that also $\DPgm_3$ is minimally pure and the cohomological comparison theorems in the proof of Theorem \ref{theo:cohdp4} above apply here too. We thus have that the cohomology groups of $\DPgm_3$ as representations of $S_6$ are as given
in Table~\ref{dp3s6table}.

\begin{table}[ht]
    \begin{equation*}
        \begin{array}{|r|ccccccccccc|}
 \hline
 &&&&&&&&&&&\\[-1em]
        \, & s_{6} & s_{1^6} & s_{2,1^4} & s_{5,1} & s_{2^3} & s_{3^2} & s_{2^2,1^2} & s_{4,2} & s_{3,1^3} & s_{4,1^2} & s_{3,2,1} \\
 &&&&&&&&&&&\\[-1em]
        \hline
 &&&&&&&&&&&\\[-1em]
\HH^0 & 1 & 0 & 0 & 0 & 0 & 0 & 0 & 0 & 0 & 0 & 0 \\
\HH^1 & 1 & 0 & 0 & 0 & 0 & 1 & 0 & 1 & 0 & 0 & 0 \\
\HH^2 & 0 & 0 & 0 & 1 & 0 & 1 & 0 & 1 & 1 & 2 & 2 \\
\HH^3 & 0 & 0 & 0 & 1 & 1 & 1 & 2 & 2 & 3 & 4 & 4 \\
\HH^4 & 1 & 1 & 1 & 1 & 3 & 3 & 2 & 2 & 2 & 2 & 2 \\
 \hline
\end{array}
    \end{equation*}
    \caption{$\HH^i(\DPgm_3)$ as a representation of $S_6$, where
    $s_{\lambda}$ denotes the irreducible representation corresponding to the partition $\lambda$.}
    \label{dp3s6table}
\end{table}

The general Del Pezzo surface of degree $3$ does not have any automorphisms
so the action of $W(E_6)$ on $\DPgm_3$ does not factor as in the case
of $\DPgm_4$. There are many representations of $W(E_6)$ restricting
to the $S_6$-representations given in Table~\ref{dp3s6table} so we need more
information in order to deduce the correct ones for Theorem \ref{theo:cohdp3}.

\section{Anticanonical sections of cubics and quartics}\label{section:anticanonicalsections}

As explained in the introduction, the approach to computing the structure of $\rm{H}^i(\DPgm_d)$ as a $W(\Phi_d)$-representation comes from first approximating it by understanding the $S_r$-action on $\DPgm_d$ given by permuting the set of points in $\PP^2$ blown up to obtain a Del Pezzo surface $X$, via point counts in Section \ref{sectionpointcounts},
and later by completing the picture by studying the cohomology, via arrangements, of various covers of $\DPgm_d$. In this section we describe these covers which will be loci inside the moduli space of geometrically marked Del Pezzo surfaces of degree $d$ along with a singular point of an anticanonical section $A\in|-K_X|$.

Note that a smooth anticanonical section $A$ of a Del Pezzo surface has genus one from the adjunction formula. We restrict from now on to
the cases $d=3$, where $X$ is anticanonically embedded into $\PP^3$ as a smooth cubic surface, and $d=4$, where $X$ is anticanonically embedded into $\PP^4$ as the smooth intersection of two smooth quadrics.
We require an analysis of the possible singularities an anticanonical section $A\in|-K_X|$ can have. We do not claim any originality here but include proofs of statements for lack of a precise reference.

\begin{Lemma}\label{anticanonicalnoreducedcpts}
    Let $X$ be a smooth Del Pezzo surface of degree $d=3$ or $4$. Then a general anticanonical section is smooth. All components of all sections are reduced and a reducible one can only have smooth rational components.
\end{Lemma}
\begin{proof}
    The first statement follows from Bertini's Theorem. Take $A=\sum a_i C_i\in|-K_X|$. As $X$ is anticanonically
    embedded and $-K_X.A=d$, there can be at most 4 irreducible components of $A$. We can choose a projection
    $\pi:X\to\PP^2$ which does not contract any components of $A$. As $A.L=3$ for $L=\pi^*\OO(1)$, if $a_i\geq2$ for some $i$, the cubic $\pi(A)$ must consist of a non-reduced component union another line. As $A.E=1$ for all lines in $X$, the $9-d$ points blown up by $\pi$ must lie on $\pi(A)$ which would force three of them to lie on a line contradicting that they are in general position. The final statement is \cite[III.3.2.3]{kollar}.
\end{proof}


\begin{Proposition}\label{propanticanonical}
    Let $X$ be a smooth Del Pezzo surface of degree $d=3,4$ and let $A=\sum a_iC_i\in|-K_X|$ be the decomposition into irreducibles of an anticanonical section. Then $A$ and its Jacobian $J=\Pic^0(A)$ are either smooth and isomorphic, or one of the following happens, where case (7) happens only for $d=4$.
    \begin{enumerate}
        \item $A$ is an irreducible nodal curve and $J=\GG_m$,
        \item $A$ is an irreducible cuspidal curve and $J=\GG_a$,
        \item $A$ consists of two $\PP^1$'s meeting transversely at two points and $J=\GG_m$,
        \item $A$ consists of two $\PP^1$'s meeting tangentially and $J=\GG_a$,\label{tangentcase}
        \item $A$ consists of three $\PP^1$'s meeting in three distinct points and $J=\GG_m$,\label{trianglecase}
        \item $A$ consists of three $\PP^1$'s passing through a single point and $J=\GG_a$,\label{Eckcase}
        \item $A$ consists of four $\PP^1$'s in a square configuration and $J=\GG_m$, \label{squarecase}
    \end{enumerate}
\end{Proposition}
\begin{proof}
    The computation of the Jacobian of the above curves is standard and can be deduced for example from \cite[\S 5.B]{harrismorrison}. From Lemma \ref{anticanonicalnoreducedcpts} we know that all $a_i=1$ and that if $A$ has more than one component then all components are isomorphic to $\PP^1$. Moreover, there cannot be more than 4 irreducible components. In degree 3 there can not be more than 3 irreducible components as every anticanonical section is a plane cubic. In degree 4, if a section has 4
    irreducible components, these must be lines. By choosing a basis $L,E_1,\ldots, E_5$ of the Picard group, we note that for example $E_1,E_2, 2L-\sum E_i, L-E_1-E_2$ are four lines in configuration (\ref{squarecase}), so this case always occurs in degree 4. An inspection shows that there are no triangle configurations of lines contained in degree 4, there can however be triangle configurations of 3 smooth $\PP^1$'s: for example $L-E_1-E_2, L-E_3-E_4, L-E_5$ are three
    smooth rational curves, two of which are lines, giving configuration (\ref{trianglecase}), and moreover these could even be chosen so that all three curves meet at one point, giving configuration (\ref{Eckcase}). The generic cubic surface does not contain an Eckardt point (i.e.\ the intersection of three lines at a point), but special cubics do. On the other hand we saw that in degree 4 there are no Eckardt points as there are no triangle configurations of lines.

    That the various other configurations occur as stated can be seen for example by projecting from a line $\ell$ (resp. a 2-plane containing two lines in $X$) onto $\PP^1$ in degree 3 (resp. degree 4). The fibres will be plane conics that can degenerate to the union of two lines, and in degree 3 the restriction of the projection $\ell\to\PP^1$ will be a degree 2 map, hence ramified at two points, in which case the conic will be tangent to $\ell$. Similarly in degree 4, the projection $X\to\PP^1$ from a plane containing a line $\ell$ has generic fibre a twisted cubic and ramification points of the map $\ell\to\PP^1$ correspond to tangency as in configuration (\ref{tangentcase})

    An elliptic pencil of hyperplanes, with central fibre an anticanonical section containing one of the lines in $X$ will not be relatively minimal in the sense of \cite[\S V.7]{bhpv}, but by degree considerations, the Kodaira classification of singularities of minimal elliptic fibrations and the discussion above one sees that an anticanonical section of degree $d$ embedded in $\PP^{d-1}$ can only be one of the ones listed in the statement.
\end{proof}

\begin{Definition}
Denote by $\DP_d^{\rm a}$ the coarse moduli space of tuples $(X,A,p)$ where $A$ is a singular anticanonical section of a smooth Del Pezzo surface $X$ of degree $d$ and $p\in A$ a singular point.
\end{Definition}

Disregarding scheme structure for the time being, from the above Lemma we have the following decompositions
\begin{eqnarray*}
    \DP_3^{\rm a} &=& \DP_3^{\rm n} \sqcup \DP_3^{\rm c} \sqcup \DP_3^{\rm 2n} \sqcup \DP_3^{\rm tn} \sqcup \DP_3^{\rm 3n} \sqcup \DP_3^{\rm tp}, \\
    \DP_4^{\rm a} &=& \DP_4^{\rm n} \sqcup \DP_4^{\rm c} \sqcup \DP_4^{\rm 2n} \sqcup \DP_4^{\rm tn} \sqcup \DP_4^{\rm 3n} \sqcup \DP_4^{\rm tp} \sqcup \DP_4^{\rm 4n},
\end{eqnarray*}
where $\rm{n,c,2n,tn,3n,tp}$ resp. $\rm{4n}$ describe the type (and number) of singularities of $A$ as appearing in configurations $(1)-(7)$ of Proposition \ref{propanticanonical} respectively, for example the locus $\DP_3^{\rm 2n}$ consists of tuples $(X,A,p)$ where $A$ has two nodes and $p$ is one of them, $\DP_3^{\rm tn}$ is the locus where $A$ has a tacnode at $p$ whereas for $\DP_3^{\rm tp}$, $A$ has a triple point namely the union of three curves through a point.

\begin{Remark}
The forgetful morphism $f:\DP_d^{\rm a}\to\DP_d$, for $d=3,4$, has $(d-1)$-dimensional fibres and its restriction to one of the loci may not be proper, as for example a nodal section can degenerate to a cuspidal one. The fibres of $f$ when restricted to $\DP_3^{*}$ for $*=\rm{c,2n}$ are 1-dimensional (see \cite[Ex 7.3(iv)]{reid} for the cuspidal case), whereas for $*=\rm{tn,3n}$ the restriction of $f$ to $\DP_3^{*}\sqcup \DP_3^{\rm tp}$ is finite surjective of degree 54 (as there are two conics tangent to each of the lines as seen by projecting from the line) and 135 respectively (as there are 45 tritangent trios on each smooth cubic by \cite[9.1.8]{dolgachev}).
\end{Remark}

\begin{Lemma}\label{auts}
    The following statements hold.
    \begin{enumerate}
        \item A general cubic surface has no automorphisms.
        \item The generic smooth cubic surface $X$ containing an Eckardt point $p\in X$ has exactly one Eckardt point and $\operatorname{Aut}(X)\cong\ZZ/2\ZZ$ is generated by the harmonic homology induced by $p$. This acts via an element of type $2A$ of the Weyl group.
        \item The general $X$ in $\DP^{*}_3$ where $*=\rm{n, c,2n,tn,3n}$ has $\Aut(X)=0$.
        \item A general degree four Del Pezzo $X$ has $\Aut(X)\cong (\ZZ/2\ZZ)^4$, whose induced action on $K_X^\perp$ is generated by $r_{\alpha_1}\circ r_{\alpha_i}$, for $i=2,3,4,5$, where $r_{\alpha_i}$ are reflections in the canonical root basis.
    \end{enumerate}
\end{Lemma}
\begin{proof}
    From \cite{tu}, in the four dimensional moduli space $\mathcal{M}$ of cubics, the locus of surfaces having an Eckardt point is an irreducible divisor whereas the locus with two Eckardt points has codimension two and two irreducible components. From \cite[Theorem 9.5.8]{dolgachev}, a general such surface with an Eckardt point has exactly one automorphism, namely the one of \cite[Proposition 9.1.13]{dolgachev}. Such an automorphism is also called a reflection, and the induced element is of type $2A$, using the notation of \cite{conwayetal},
    according to \cite[\S 1.2]{dd}. The table in \cite[Theorem 9.5.8]{dolgachev} also proves the first and third claims.

     The final statement follows from \cite[\S 8.6.4]{dolgachev}. The action of $\Aut(X)$ on the space $K_X^\perp$ is easy to describe for a general $X$: for $L, E_1,\ldots, E_5$ a usual geometric basis for the Picard group, the four generators of $(\ZZ/2\ZZ)^4$ are the elements $r_{\alpha_1}\circ r_{\alpha_i}$, $i=2,3,4,5$, where $\alpha_i$ form the canonical root basis
    \begin{equation*}
        \alpha_1=L-E_1-E_2-E_3 \text{ and }\alpha_i=E_{i-1}-E_i\text{ for }i=2,3,4,5.
    \end{equation*}
    and $r_{\alpha_i}(v)=v+(\alpha_i.v)\alpha_i$ is the reflection in $\alpha_i$.
\end{proof}

\section{Hyperplane and toric arrangements}\label{section:arrangements}

In this section, we follow ideas of Looijenga \cite{looijenga} to establish isomorphisms between last section's moduli spaces of geometrically marked Del Pezzo surfaces along with a singular point of an anticanonical section and
complements of toric and hyperplane arrangements. As these isomorphisms follow in most cases verbatim, and in some with minor modifications, from the ones in \cite{looijenga}, in the following subsection we only outline the construction of the various arrangement spaces so as to fix notation and give the reader an idea of the arguments.

\subsection{Moduli of singular anticanonical sections}
Let $(X,A,p)\in\DP_3^{\rm n}$ be a marked cubic surface with a nodal anticanonical section as in the previous section. Consider the restriction homomorphism $$r_A:K_X^\perp\to J(A).$$ In Proposition \ref{propanticanonical} we gave an isomorphism $J(A)\cong\CC^*$. In fact there are precisely two such group homomorphisms, one being the inverse of the other. Composing with one of them, the map $r_A$ determines an element of $\Hom(K_X^\perp,\CC^*)\cong(\CC^*)^6$.
To be more precise let $L_{E_6}$ be the $E_6$ lattice (see e.g.\ \cite[8.2.2]{dolgachev}) which is, by our choice of a geometric marking, isometric to $K_X^\perp$. If we denote now by $T$ the torus $\Hom(L_{E_6},\CC^*)$, we have a $W(E_6)$-equivariant homomorphism
\begin{eqnarray*}
\DP_3^{\rm n} &\to& (\ZZ/2\ZZ)\backslash T \\
(X,A,p) &\mapsto& r_A.
\end{eqnarray*}
The smooth locus $A_{\rm{sm}}\subset A$ is isomorphic to $\Pic^1(A)$, which is in turn a $J(A)$-torsor.
Note now that if $\alpha\in L_{E_6}$ is a root, then $r_A(\alpha)\neq1$ (where here $1$ corresponds to the trivial element $\OO_A\in J(A)$) since otherwise the 6 points giving
the blowup description of $X$ would not be in general position, e.g.\ $r_A(e_1-e_2)=1$ if two of the points are the same,
$r_A(\ell-e_1-e_2-e_3)=1$ if the three points are collinear or
$r_A(2\ell -e_1- \cdots -e_6)$ if the six points lie on a conic. In particular we see that the kernel of the
map $f_\alpha:T\to\CC^*:\chi\mapsto \chi(\alpha)$, which is a hypertorus in $T$, is not contained in the image of $\DP_3^{\rm n}$ for any root $\alpha$. Note that the roots $\alpha,-\alpha$ both give the same hypertorus, so we need only consider positive roots, the set of which we denote by $\mathcal{R}^+$. For the computations, we note that $\ker f_\alpha=\{\chi\in T: \chi(v)=\chi(\operatorname{refl}_\alpha(v))\}$ is the fixed locus of the reflection in $\alpha$.
Finally, we denote the complement by $$T_{E_6}=T\setminus\bigcup_{\alpha\in\mathcal{R}^+} \ker f_\alpha.$$

If $(X,A,p)\in\DP_3^{\rm c}$ so that $A$ has a cusp at $p$, an isomorphism $J(A)\cong\CC$ is defined up to $\CC^*$, and so if $V=\Hom(L_{E_6},\CC)$ we get a map $\DP_3^{\rm c}\to\PP(V)$. Analogously to the nodal case, we have hyperplanes $\ker f_\alpha\subset V$ for $\alpha$ a positive root, for which we define $V_{E_6} = V\setminus\bigcup_\alpha \ker(f_\alpha)$ and note that the image of $\DP_3^{\rm c}$ lies inside $V_{E_6}/\CC^*$ which for ease of notation and to emphasise that it is a complement in a projective space we denote by $\PP'(V_{E_6})$.

If $(X,A,p)\in\DP_3^{\rm 2n}$ so that $A$ has two irreducible components $F_1,F_2$ meeting transversely at two points one of which is $p$ then\footnote{Note that there is a small mistake here in \cite[1.12]{looijenga} which does not affect the computations.} $J(A)\cong\CC^*$, and like in the nodal case, we must quotient by $\ZZ/2\ZZ$ as there are two such isomoprhisms. For example, by picking a suitable geometric marking we have $F_1=2L-\sum_{i=1}^5E_i$ and
$F_2=L-E_6$. One checks now that $\langle F_1, F_2\rangle^\perp\subset K_X^\perp$ is isometric to the $D_5$ lattice - call $f_i$ the image of $F_i$. Fixing $f_1$ with $f_1^2=-1$ the corresponding class $f_2$ so that $f_1+f_2=-k$ will have $f_2^2=0$ and can be represented by a pencil of rational curves (those residual in the 2-planes containing $F_1$). Every line $f_1$ will contribute one different torus. Note though that as we vary along the pencil, we obtain (at least generically) two distinct singular points $p_1,p_2\in A$, and the triples $(X,A,p_i)$ for $i=1,2$ both induce the same homomorphism $r_A$, so we must quotient the space $\DP_3^{\rm 2n}$ by the natural $\ZZ/2\ZZ$-action swapping the two singular points of $A$. In other words we will obtain the unmarked 2-nodal locus as a toric arrangement. Denote
$T(f_1)=\Hom(\langle f_1, f_2\rangle^\perp,\CC^*)$ and as in previous cases, we want to avoid characters which are trivial on positive roots, so we denote $T_{D_5}(f_1)=T(f_1)\setminus\cup_\alpha\ker f_\alpha$.

If $(X,A,p)\in\DP_3^{\rm tn}$ so that $A$ has two irreducible components $F_1,F_2$ meeting non-transversely at one point $p$ with multiplicity two, then the classes of the $F_i$ give an isometry $\langle f_1,f_2\rangle^\perp=L_{D_5}$ as in the case $\DP_3^{\rm 2n}$, so we obtain (following the notation of the previous two paragraphs) a point in $\PP'(V_{D_5}(e))$. Note that for every fixed $(-1)$-curve $F_1$, there are precisely two smooth rational curves of class $f_2$ which meet $F_1$ tacnodally, as can be seen by projecting from $F_1$.

If $(X,A,p)\in\DP_3^{\rm 3n}$ (resp. $\DP_3^{\rm tp}$), $A=F_1+F_2+F_3$ is the sum of three lines (e.g.\ $E_1,L-E_1-E_2,2L-\sum E_i+E_2$), then the classes $f_i$ satisfy $\langle f_1,f_2,f_3\rangle^\perp = L_{F_4}$ (cf.\
\cite[9.1.10]{dolgachev}). Note that $W(E_6)$ acts transitively on the set of tritangent trios \cite[9.1.9]{dolgachev} and a list of them (in terms of a chosen basis) can be found loc.\ cit. Analogously to the cases above we obtain that $r_A$ lies in $T_{F_4}(f_1)$ (resp. $\PP'(V_{F_4}(f_1))$). There is a $\ZZ/3\ZZ$-action permuting the role of $f_1,f_2,f_3$ which we must eventually quotient by.

In the following, denote by $\widehat{\DP_3^{\rm 2n}}$ (resp. $\widehat{\DP_3^{\rm 3n}}$) the quotient of the space by the natural $\ZZ/2\ZZ$-action (resp. $\ZZ/3\ZZ$) permuting the singular points of $A$. Analogous to \cite[1.7-1.15]{looijenga} we now obtain.
\begin{Theorem}\label{theo:loo}
    We have the following $W(E_6)$-equivariant isomorphisms
    \begin{align*}
        \DP_3^{\rm n}    \cong& (\ZZ/2\ZZ)\backslash T_{E_6}, & \DP_3^{\rm tn} &\cong \bigsqcup_{e^2=-1}\PP'(V_{D_5}(e)), \\
        \DP_3^{\rm c}    \cong& \PP'(V_{E_6}),              & \widehat{\DP_3^{\rm 3n}} &\cong \bigsqcup_{(f_1,f_2,f_3)}(\ZZ/2\ZZ)\backslash T_{F_4}(f_1)/(\ZZ/3\ZZ),  \\
        \widehat{\DP_3^{\rm 2n}} \cong& \bigsqcup_{e^2=-1}(\ZZ/2\ZZ)\backslash T_{D_5}(e),  & \DP_3^{\rm tp} &\cong  \bigsqcup_{(f_1,f_2,f_3)}\PP'(V_{F_4}(f_1))/(\ZZ/3\ZZ) .
    \end{align*}
\end{Theorem}
\begin{proof}
    We have already constructed the corresponding morphisms in the above paragraphs. We construct now the inverse in the case $\DP_3^{\rm 3n}$ as the case of a triple point is analogous, and the remaining ones have already been covered in the work of Looijenga \cite[1.8, 1.11, 1.13, 1.15]{looijenga} for degree 2 Del Pezzo surfaces, but follow mutatis mutandis in our case.

    Assume we are given a $\chi\in (\ZZ/2\ZZ)\backslash T_{F_4}(f_1)$, that is to say an element of $\Hom(L_{F_4},\CC^*)$ where we identify
    $L_{F_4}=\langle f_1,f_2,f_3\rangle^\perp$ for $f^2_i=-1$ so that $\chi(\alpha)\neq1$ for all positive roots $\alpha$ of $F_4$, and an action of $\ZZ/3\ZZ$ on $L_{F_4}$ permuting $f_1,f_2,f_3$. We want to construct a cubic surface with a geometric marking from this data. Given a geometric marking, the $f_i$ will correspond for example to the lines $2L-\sum_{i=1}^5 E_i, L-E_5-E_6, E_5$ respectively on a cubic surface, but recall that the Weyl group acts transitively on the set of tritangent trios.

    Consider an abstract curve $A$ which is a triangle configuration of three smooth rational curves $F_1,F_2,F_3$, along with a $\ZZ/3\ZZ$-action on $A$ which permutes the irreducible components through the marked singularity. Fix a group isomorphism $\phi:\CC^*\to J(A)$, and note that from Riemann-Roch we have an isomorphism $\psi:\Pic^1(A)\to A_{\rm{sm}}$. Pick a general point $P_1\in F_1\cap A_{\rm{sm}}$ and define $P_2$ to be the point $\psi(\OO_A(P_1)\otimes\phi(\chi(e_2-e_1))$. By acting by the cyclic group we can ensure $P_2\in F_1\cap A_{\rm sm}$. Similarly define two more points on $F_1\cap A_{\rm sm}$, $P_i=\psi(\OO_A(P_{i-1})\otimes\phi(\chi(e_i-e_{i-1}))$ for $i=3,4$. Finally, define $P_6\in F_2\cap A_{\rm sm}$ to be the image of $P_1$ under the $\ZZ/3\ZZ$-action. The linear system $|\OO_A(P_1+P_2+P_6)\otimes\phi(\chi(\ell-e_1-e_2-e_6))|$ defines a morphism $h:A\to\PP^2$ which embeds $F_1$ as a conic, $F_2$ as a line and contracts $F_3$. Our desired cubic surface will now be the blowup of $\PP^2$ at the images of the five points $P_i$ and at the point which is the image of $F_3$.

    What remains to be checked is that the six points in the plane - call them $P_i$ again from now on for simplicity - are in general position. Since $\chi(e_{i+1}-e_i)\neq1$ (as these are roots in $F_4$) for $i=1,2,3$, the points $P_1,\ldots,P_4$ are distinct, as is $P_6$ as it lies on a different component. Since these previous 5 points are smooth points of the image of $A$, they all differ from $P_5=h(F_3)$. The condition that all 6 of them do not lie on a conic is automatic as $P_1,\ldots,P_5$ already lie on the conic $h(F_1)$ whereas $P_6$ does not. Similarly, no three of $\{P_1,\ldots,P_5\}$ are collinear as they all lie on a conic. From the definition of the $P_i$ for $i\leq4$ we find that $|\OO_A(P_i+P_j+P_6)\otimes\phi(\chi(\ell-e_i-e_j-e_6))|$ is, for $1\leq i<j\leq4$, the same linear system as that inducing $h$ so $P_i,P_j,P_6$ are not collinear as we have twisted by the non-trivial value under $\chi$ of the root $\ell-e_i-e_j-e_6$. That $P_i,P_5,P_6$ are not collinear for $i\leq4$ is automatic as $P_5,P_6$ lie on the line $h(F_2)$ whereas the remaining $P_i$ do not.
\end{proof}

\subsection{Cohomology of arrangement complements}

We recall the definition, due to Dimca--Lehrer \cite[\S 3]{dimcalehrer}, that an
irreducible complex variety $X$ is \emph{minimally pure} if each cohomology group
$\HH_c^i(X, \mathbb{C})$ is a pure Hodge structure of weight
$2i-2\mathrm{dim}(X)$.
In a minimally pure variety, one can define what it means to be a \textit{minimally pure arrangement}, examples of which are toric
arrangements and arrangements of hyperplanes.
An important feature of minimally pure arrangements is that their complements
are minimally pure. There is also an explicit formula for the cohomology
of the complement of a minimally pure arrangement in terms of the intersection
poset. In order to state the result, we define
the equivariant Poincar\'e polynomial $P(X,t)$ at an automorphism $\sigma$ of
$X$ of finite order as
\begin{equation*}
    P(X,t)(\sigma) =
    \sum_{i \geq 0} \mathrm{Tr} \left(\sigma, \HH_c^i(X, \mathbb{C}) \right)t^i.
\end{equation*}

\begin{Theorem}[Macmeikan \cite{macmeikan}]\label{macmeikan}
 Let $\mathcal{A}=\{A_i\}_{i \in I}$ be a minimally pure arrangement
 in a minimally pure variety $X$ and let
 \begin{equation*}
     X_{\mathcal{A}} = X \setminus \bigcup_{i \in I} A_i
 \end{equation*}
 denote the complement. Let $\sigma$ be an automorphism of $X$ of finite order
 which stabilises $\mathcal{A}$ as a set and let $\mathcal{L}^{\sigma}$
 be the poset of intersections of elements of $\mathcal{A}$ which are fixed by
 $\sigma$ and let $\mu$ be its M\"obius function. Then
 \begin{equation*}
     P(X_{\mathcal{A}},t)(\sigma) =
     \sum_{Z \in \mathcal{L}^{\sigma}}
     \mu(Z)(-t)^{\mathrm{codim}(Z)}P(Z,t)(\sigma).
 \end{equation*}
\end{Theorem}

The above result is a generalization of the Orlik--Solomon formula for
hyperplane arrangements.
The poset $\mathcal{L}^{\sigma}$ quickly becomes very large (e.g.\ the poset
of the toric arrangement associated to $E_6$ contains $5079$ elements) and
computations by hand are therefore rarely an option.
There are however efficient algorithms and implementations in
the case of hyperplane arrangements, due to Fleischmann and Janiszczak
\cite{fleischmannjaniszczak_hyp}, and in the case of toric arrangements, due
to the first author, \cite{bergvalltor}. This has allowed us to compute
Tables~\ref{nodecohtable}-\ref{Ecohtable}. We should remark that the contents of Table~\ref{cuspcohtable} can be found, in a slightly different form, in \cite{fleischmannjaniszczak_hyp} and Table~\ref{nodecohtable} can be found in \cite{bergvalltor}.
Information relevant for the computation of Table~\ref{Ecohtable} resp. Table~\ref{tlcohtable} can also be found in \cite{fleischmannjaniszczak_hyp}
resp. \cite{bergvalltor}.

We recall the following construction, due to Looijenga~\cite{looijenga}, which
allows the computation of some unions of loci of the type $\DP_d^{*}$ as in Section \ref{section:anticanonicalsections}.
Let $\mathcal{A}$ be an arrangement of codimension $1$ subtori in a torus
  $T$ such that each element of $\mathcal{A}$ passes through the identity of $T$ and let $D=\bigcup_{Z \in \mathcal{A}} Z$. Let $T'$ be the blowup of
  $T$ in the identity and let $D'$ be the strict transform of $D$. Let $V$
  be the tangent space of the identity in $T$. We may then identify $\PP(V)$
  with the exceptional divisor in $T'$. Under this identification, let
  $D_V=D'\cap \PP(V)$. We then have
  \begin{equation*}
      T'\setminus D' = \left( T \setminus D \right) \sqcup \left( \PP(V) \setminus D_V \right).
  \end{equation*}

\begin{Lemma}(Looijenga, \cite[Lemma 3.6]{looijenga})
\label{blowuplemma}
  Let $\Gamma$ be a group stabilizing $\mathcal{A}$ as a set. There is then
  a $\Gamma$-equivariant exact sequence of mixed Hodge structures
  \begin{equation*}
     0 \to \HH^i(T'\setminus D') \to
     \HH^i(T \setminus D) \to
     \HH^{i-1}( \PP(V) \setminus D_V)(-1) \to 0
  \end{equation*}
\end{Lemma}

If $\mathcal{A}$ is the toric arrangement associated to a root system,
then $\PP(V) \setminus D_V$ is the projectivization of the corresponding
arrangement of hyperplanes. Thus, in this case one can obtain
the cohomology of $T'\setminus D'$ by computing the cohomology of the complement of a toric arrangement and the cohomology of the complement of a hyperplane arrangement.

Let $\Phi$ be a root system. Let $T_{\Phi}$ be the complement of the toric arrangement associated to $\Phi$, let $V_{\Phi}$ be the complement of the hyperplane arrangement associated to $\Phi$ and let $T'_{\Phi}$ be the
variety obtained from $T_{\Phi}$ by blowing up the ambient torus in the identity, as described above. Completely analogously to \cite[Propositions 1.17, 1.18]{looijenga} we have isomorphisms as follows.

\begin{Proposition}
\label{unionprop}
 There are $W(E_6)$-equivariant isomorphisms
 \begin{equation*}
     \begin{array}{lcl}
     \DP_3^{\rm n} \bigsqcup \DP_3^{\rm c} & \cong & (\ZZ/2\ZZ)\setminus T_{E_6}' \\
     \widehat{\DP_3^{\rm 2n}} \bigsqcup \DP_3^{\rm tn}  & \cong &
     \bigsqcup_{W(E_6)/W(D_5)} (\ZZ/2\ZZ)\setminus T_{D_5}' \\
     \widehat{\DP_3^{\rm 3n}} \bigsqcup \DP_3^{\rm tp} & \cong &
      \bigsqcup_{W(E_6)/W(F_4)} (\ZZ/2\ZZ)\setminus T_{F_4}'/(\ZZ/3\ZZ).
      \end{array}
 \end{equation*}
\end{Proposition}

If we recall Theorem \ref{theo:loo}
and apply Lemma~\ref{blowuplemma} to Proposition~\ref{unionprop}
we obtain the following.

\begin{Corollary}
 The cohomology groups of $\DP_3^{\rm n} \sqcup \DP_3^{\rm c},
 \DP_3^{\rm 2n} \sqcup \DP_3^{\rm tn}, \DP_3^{\rm 3n} \sqcup \DP_3^{\rm tp}$ are all pure
 of type $(i,i)$.
 Moreover, in the representation ring of $W(E_6)$,
 the following equalities hold
 \begin{equation*}
 \begin{array}{lcl}
      \HH^i(\DP_3^{\rm n} \bigsqcup \DP_3^{\rm c}) & = &
      \HH^i(\DP_3^{\rm n}) - \HH^{i-1}(\DP_3^{\rm c}) \\
      \HH^{i}(\DP_3^{\rm 2n} \bigsqcup \DP_3^{\rm tn}) & = &
      \HH^{i}(\DP_3^{\rm 2n}) - \HH^{i-1}(\DP_3^{\rm tn}) \\
      \HH^{i}(\DP_3^{\rm 3n} \bigsqcup \DP_3^{\rm tp}) & = &
      \HH^{i}(\DP_3^{\rm 3n}) - \HH^{i-1}(\DP_3^{\rm tp}).
\end{array}
 \end{equation*}
\end{Corollary}
Note now the following theorem of Wells which we will use repeatedly.
\begin{Theorem}(Wells, \cite{wells})\label{wells}
Let $\pi:X\to Y$ be a surjective proper morphism of complex quasiprojective smooth varieties. Then $\pi$ induces an injection $\pi^*:\HH^i(Y,\CC)\to\HH^i(X,\CC)$.
\end{Theorem}

\begin{Corollary}
\label{dp3puritycor}
 The cohomology groups $\HH^i(\DP_3,\CC)$ are pure of type $(i,i)$
 and $\DP_3/\overline{\FF}_q$ is minimally pure in the sense of Definition \ref{defn:minpure}.
\end{Corollary}
\begin{proof}
    The forgetful morphism $\DP_3^{\rm 3n} \sqcup \DP_3^{\rm tp} \to
    \DP_3$ is finite and thus induces an injection
    of mixed Hodge structures $\HH^i(\DP_3) \to  \HH^{i}(\DP_3^{\rm 3n} \sqcup \DP_3^{\rm tp})$.
    Since the latter is pure
    of type $(i,i)$, the same is true for
    $\HH^{i}(\DP_3)$.
    Similarly, we get an injection $\rm{H}_{\et,c}^i(\DP_3,\Qell) \to
    \rm{H}_{\et,c}^i(\DP_3^{\rm 3n} \sqcup \DP_3^{\rm tp},\Qell)$.
    Since $\Frob$ acts with all eigenvalues equal to $q^{i-4}$ on the latter,
    the same is true for its action on the subspace
    $\rm{H}_{\et,c}^i(\DP_3,\Qell)$.
\end{proof}

\section{Proof of Theorem \ref{theo:cohdp3}}\label{section:mainproof}

The following will be used repeatedly.
\begin{Lemma}(Looijenga, \cite[Lemma 4.1]{looijenga})
\label{gysinlemma}
Let $X$ be a variety of pure dimension and let $Y \subset X$ be a hypersurface.
Assume furthermore that both $X$ and $Y$ are rational homology manifolds.
Then there is a Gysin exact sequence of mixed Hodge structures
\begin{equation*}
\cdots \to \HH^{k-2}(Y)(-1) \to \HH^{k}(X) \to \HH^{k}(X\setminus Y) \to \HH^{k-1}(Y)(-1) \to \cdots
\end{equation*}
\end{Lemma}

As an application of the above and our analysis in Section \ref{section:anticanonicalsections}, we obtain the following key Lemma.

\begin{Lemma}\label{comparelemma}
There are $W(E_6)$-equivariant inclusions of mixed Hodge structures
from $\HH^i(\DP_3)$ to the $i$-th cohomology of any of the following spaces $
\DP_3^{\rm n} \sqcup \DP_3^{\rm c},\
\DP_3^{\rm 2n} \sqcup \DP_3^{\rm tn},\
\DP_3^{\rm 3n} \sqcup \DP_3^{\rm tp},\
\DP_3^{\rm n},\
\DP_3^{\rm c},\
\DP_3^{\rm 2n},\
\DP_3^{\rm tn},\
\DP_3^{\rm 3n}.
$
\end{Lemma}
\begin{proof}
   Some of the cases are explained above. We give a proof only in the case of $\HH^{i}(\DP_3^{\rm 2n} \sqcup \DP_3^{\rm tn})$, the others being similar.

   Let $U=\DP_3^{\rm 2n} \sqcup \DP_3^{\rm tn}$,
   $Y=\DP_3^{\rm 3n} \sqcup \DP_3^{\rm tp}$ and let $X=U \sqcup Y$.
   We apply Lemma~\ref{gysinlemma} and get an exact sequence of
   mixed Hodge structures
   \begin{equation*}
 \cdots \to \HH^{i-2}(Y)(-1) \to \HH^{i}(X) \to \HH^{i}(U) \to \HH^{i-1}(Y)(-1) \to \cdots  .
\end{equation*}
Since both $\HH^{i}(Y)$ and $\HH^{i}(U)$ are pure of type $(i,i)$ we
see that $\HH^i(X)$ can consist of at most two parts - one of type $(i,i)$
coming from $\HH^{i}(U)$
and one of type $(i-1,i-1)$ coming from $\HH^i(Y)$. The forgetful morphism $X \to \DP_3$ is proper so from Theorem \ref{wells} we get an inclusion of mixed Hodge structures
\begin{equation*}
    \HH^i(\DP_3) \to \HH^i(X).
\end{equation*}
But by Corollary~\ref{dp3puritycor} we know that $\HH^i(\DP_3)$
is pure of type $(i,i)$ so the image of the above injection must lie in
the $(i,i)$ part of $\HH^i(X)$ which is in turn contained in $\HH^i(U)$.
\end{proof}

By using the above inclusion to compare the $S_6$-equivariant information
about $\HH^i(\DP_3)$ with the $W(E_6)$-equivariant information about
$\HH^i(\DP_3^{\rm c})$ we can deduce the
$W(E_6)$-equivariant structure of $\HH^0(\DP_3)$, $\HH^1(\DP_3)$ and $\HH^2(\DP_3)$. However, this information only leaves us
with two possibilities for $\HH^3(\DP_3)$ (and several for
$\HH^4(\DP_3)$). More precisely
we have that $\HH^3=\phi_{15}^5+\phi_{90}^8+\chi_{80}$ where $\chi_{80}$ is a
$80$-dimensional representation equal to one of
\begin{equation*}
    \phi_{20}^{10}+\phi_{60}^8 \quad \text{or} \quad \phi_{80}^7.
\end{equation*}
The other inclusions in cohomology
from Lemma~\ref{comparelemma} give further restrictions on the multiplicities of each irreducible representation in $\HH^i(\DP_3)$. This reduces the number of possibilities for $\HH^4(\DP_3)$ to $8$ but fails to give any new information about $\HH^3(\DP_3)$. We have $\HH^4=\phi_{10}^9+\chi_{140}$ where $\chi_{140}$ is a 140-dimensional representation equal to one of the following eight representations
\begin{equation*}
    \begin{array}{ll}
        \phi_{80}^7 + \phi_{30}^3 + \phi_{30}^{15} &  \phi_{80}^7 + \phi_{15}^{4} + \phi_{15}^5 + \phi_{15}^{16}+\phi_{15}^{17} \\
        \phi_{80}^7 + \phi_{30}^{15} + \phi_{15}^4 + \phi_{15}^5 &  \phi_{80}^7 + \phi_{30}^3 + \phi_{15}^{16} + \phi_{15}^{17} \\
        \phi_{60}^8 + \phi_{20}^{10} + \phi_{30}^{15} + \phi_{15}^{4} + \phi_{15}^{5}  
        &  \phi_{60}^8 + \phi_{20}^{10} + \phi_{30}^{3} + \phi_{15}^{16} + \phi_{15}^{17} \\
        \phi_{60}^8 + \phi_{20}^{10} + \phi_{15}^{4} + \phi_{15}^5 + \phi_{15}^{16} + \phi_{15}^{17} &  
        \phi_{60}^8 + \phi_{20}^{10} + \phi_{30}^3 + \phi_{30}^{15}.\\
    \end{array}
\end{equation*}

We now make the following simple observation. We know that
\begin{equation*}
    |(\DP_3)^{\Frob \sigma}| = \sum_{i=0}^4 \mathrm{Tr}(\sigma,\HH^i_\et(\DP_3,\Qell))(-q)^{4-i}
\end{equation*}
and we know $\mathrm{Tr}(\sigma,\HH^i_\et(\DP_3,\Qell))$ for
$i=0,1,2$ and, for each $\sigma \in W(E_6)$, we have at most two
different possibilities for $\mathrm{Tr}(\sigma,\HH^3_\et(\DP_3,\Qell))$
and at most eight different possibilities for $\mathrm{Tr}(\sigma,\HH^4_\et(\DP_3,\Qell))$.
We may thus plug in the different values for
$\mathrm{Tr}(\sigma,\HH^3_\et(\DP_3,\Qell))$ and $\mathrm{Tr}(\sigma,\HH^4_\et(\DP_3,\Qell))$ and see what we get.
The result should be a polynomial counting the number of fixed points
of an automorphism of a variety over $\FF_q$ but 
in all $8$ cases where 
$\HH^3=\phi_{15}^5+\phi_{90}^8+\phi_{20}^{10}+\phi_{60}^8$
we get negative results at $q=2,3$ and $5$. 
We therefore conclude that $\HH^3=\phi_{15}^5+\phi_{90}^8+\phi_{80}^{7}$.

To compute $\HH^4$ we observe that $\DP_3$ is a $W(D_5)$-equivariant fibration
over $\DP_4$ (by forgetting the last point in the
blow-up picture). Let $F$ denote the fiber. For a complex variety $X$, denoting $$\chi(X)(g)=\sum_{i=0}^{\dim X}(-1)^i\rm{Tr}(g, \HH^k(X,\QQ)),$$
from the above we conclude that
the Euler characteristics satisfy
\begin{equation*}
    \chi(\DP_3)(g) = \chi(\DP_4)(g) 
    \cdot \chi(F)(g).
\end{equation*}
In particular, if $\chi(\DP_4)(g)=0$ then $\chi(\DP_3)(g)=0$. 
We know the cohomology of 
$\DP_4$ as a $W(D_5)$-representation so computing
$\chi(\DP_4)(g)$ is immediate. We see that $6$ classes
are such that $\chi(\DP_4)(g)=0$ (two consisting of elements of order $2$, two consisting of elements of order $4$, one consisting of elements of order $6$ and one consisting of elements of order $12$). When considered as elements of $W(E_6)$, $4$ out of these classes contain elements coming from
$S_6$ (and we thus already know everything about them from the point count) but two (one of order 4 and the one of order 12) are ``new''. Computing the Euler
characteristic at these classes using each of our $8$ candidates for $\HH^4$ gives
that $4$ take value $\pm 8$ at the class containing elements of order $12$. After removing these candidates, our possibilities for $\chi_{140}$ are
\begin{equation*}
    \begin{array}{ll}
        \phi_{80}^7 + \phi_{30}^3 + \phi_{30}^{15} &  
        \phi_{80}^7 + \phi_{15}^{4} + \phi_{15}^5 + \phi_{15}^{16}+\phi_{15}^{17} \\
        \phi_{60}^8 + \phi_{20}^{10} + \phi_{15}^{4} + \phi_{15}^5 + \phi_{15}^{16} + \phi_{15}^{17} &  
        \phi_{60}^8 + \phi_{20}^{10} + \phi_{30}^3 + \phi_{30}^{15}\\
    \end{array}
\end{equation*}
Using the fact that $\DP_3 \to \DP_4$ is defined
over $\mathbb{Z}$ (in fact, that it is defined over $\mathbb{R}$ is enough)
we can extend the action of $W(D_5)$ to an action of $\mathbb{Z}/2\mathbb{Z} \times W(D_5)$ by letting $\mathbb{Z}/2\mathbb{Z}$ act by complex conjugation (see \cite{BergvallJFT} for details). Denote the element $(1,g)$ by $\overline{g}$.
By minimal purity, for $d=3,4$ we have
\begin{equation*}
    \chi(\DP_d)(\overline{g}) = \sum_{i=0}^{\mathrm{dim}(\DP_d)} \mathrm{Tr}(g, \HH^i(\DP_d, \mathbb{C})),
\end{equation*}
(this essentially just says that complex conjugation acts on $\HH^k$ as $(-1)^k$ which is a consequence of $\HH^k$ having pure Tate type). 
We again use the multiplicativity of the Euler characteristic
but now evaluate at $\overline{g}$. There are $8$ classes in $W(D_5)$
such that $\chi(\DP_4)(\overline{g})=0$ out of which $3$ are not from $S_6$ when viewed as elements in $W(E_6)$. Of our four candidates, $3$ give
nonzero values at one of these classes (it is not the same class for all three though). We finally conclude that
\begin{equation*}
    \HH^4 = \phi_{10}^9 + \phi_{80}^7 + \phi_{30}^3 + \phi_{30}^{15}.
\end{equation*}

\section{Quartic Del Pezzo surfaces with anticanonical curves}\label{section:quarticpairs}

 Recall the decomposition
 \begin{equation*}
\begin{array}{cc}
    \DP_4^{\rm a} =& \DP_4^{\rm n} \sqcup \DP_4^{\rm c} \sqcup \DP_4^{\rm 2n} \sqcup \DP_4^{\rm tn} \sqcup \DP_4^{\rm 3n} \sqcup \DP_4^{\rm tp} \sqcup \DP_4^{\rm 4n}.
\end{array}
\end{equation*}
In this section we will show how to compute the cohomology groups of the various loci above, using hyperplane and toric arrangements like in the cubic case in Section \ref{section:arrangements}. The resulting computations are all presented in the second Appendix.

\subsection{Irreducible anticanonical curves}
Let $S$ be a Del Pezzo surface of degree $4$. Let $A$ be an anticanonical curve.
The orthogonal complement $A^\perp\subset\Pic(S)$ gives a root system of type $D_5$.
From Lemma \ref{auts}, the automorphism group of a general $S$ is $(\ZZ/2\ZZ)^4$.
If $A$ is nodal we have $\Jac(A) \cong \CC^*$ and if $A$ is cuspidal we have
$\Jac(A) \cong \CC$. Following Looijenga \cite{looijenga} as in the proof of Theorem \ref{theo:loo}
we now have the following.

\begin{Proposition}
 There are $W(D_5)$-equivariant isomorphisms
 \begin{equation*}
     \begin{array}{lcl}
     \DP_4^{\rm n} & \cong & (\ZZ/2\ZZ)^4 \setminus T_{D_5} \\
      \DP_4^{\rm c}  & \cong & (\ZZ/2\ZZ)^4 \setminus \PP(V_{D_5}).
      \end{array}
 \end{equation*}
\end{Proposition}

Using these descriptions, the same ideas as in Section \ref{section:mainproof} and the methods developed in \cite{fleischmannjaniszczak_hyp} and \cite{bergvalltor}
we compute the cohomology groups of $\DP_4^{\rm n}$ and $\DP_4^{\rm c}$
in the second Appendix. We remark that, quite unexpectedly,
$\DP_4^{\rm c}$ has zero fourth cohomology group. We will elaborate on this in Remark \ref{rem:samecoh}.

\subsection{Anticanonical curves with two components}
If an anticanonical curve $A$ consists of two components $A_1$ and $A_2$ these components must
have classes of the following three types
\begin{equation*}
    \begin{array}{lcllcl}
(1)\,         A_1 & = & 2L-E_1-\cdots-E_5, & A_2 & = & L  \\
(2)\,         A_1 & = & 2L-E_1-\cdots-E_5+E_i, & A_2 & = & L-E_i \\
(3)\,         A_1 & = & 2L-E_1-\cdots-E_5+E_i+E_j, & A_2 & = & L-E_i-E_j.
    \end{array}
\end{equation*}
In cases (1) and (3), $\langle A_1,A_2\rangle^\perp$ is a root system of type $A_4$ while in case (2) we get a root system of type $D_4$. The Weyl group $W(D_5)$ acts transitively on the anticanonical curves giving root systems of the same type. We thus see that
$\DP_4^{\rm 2n}$ and $\DP_4^{\rm tn}$ decompose further as
\begin{equation*}
    \DP_4^{\rm 2n} = \DP_4^{\rm 2n,A_4} \sqcup \DP_4^{\rm 2n,D_4}, \quad
     \DP_4^{\rm tn} = \DP_4^{\rm tn,A_4} \sqcup \DP_4^{\rm tn,D_4},
\end{equation*}
with one component for each type of root system. Again, by mimicking Looijenga's argument as in Theorem \ref{theo:loo} we get the following.

\begin{Proposition}
 There are $W(D_5)$-equivariant isomorphisms
 \begin{equation*}
     \begin{array}{lcl}
     \DP_4^{\rm 2n,A_4} & \cong & (\ZZ/2\ZZ)^4 \setminus \bigsqcup_{W(D_5)/W(A_4)} T_{A_4} \\
     \DP_4^{\rm 2n,D_4} & \cong & (\ZZ/2\ZZ)^4 \setminus \bigsqcup_{W(D_5)/W(D_4)} T_{D_4} \\
     \DP_4^{\rm tn,A_4} & \cong & (\ZZ/2\ZZ)^4 \setminus \bigsqcup_{W(D_5)/W(A_4)} \PP(V_{A_4}) \\
     \DP_4^{\rm tn,D_4} & \cong & (\ZZ/2\ZZ)^4 \setminus \bigsqcup_{W(D_5)/W(D_4)} \PP(V_{D_4}).
      \end{array}
 \end{equation*}
\end{Proposition}

We give the cohomology groups of the above moduli spaces in the second Appendix.

\begin{Remark}\label{rem:samecoh}
We would like to point out that the cohomologies of $\DP_4^{\rm c}$ and $\DP_4^{\rm{tn, A_4}}$ in Table~\ref{dp4table} are
exactly the same. This may look a bit curious but has a rather simple explanation. To compute the cohomology of $\DP_4^{\rm c}$ we first compute
the cohomology of the complement of the projectivised hyperplane arrangement associated to $D_5$ and then take invariants with respect to $(\ZZ/2\ZZ)^4$ in order to take the quotient. The quotient of of $W(D_5)$ by  $(\ZZ/2\ZZ)^4$ is
isomorphic to $W(A_4)$. To compute the cohomology of $\DP_4^{\rm tn, A_4}$ we
first compute
the cohomology of the complement of the projectivised hyperplane arrangement associated to $A_4$ and then induce up to $W(D_5)$ before taking $(\ZZ/2\ZZ)^4$-invariants. Thus, in light of Frobenius reciprocity, the
equality of the tables is not as remarkable as it would seem at first sight.
\end{Remark}

\subsection{Anticanonical curves with three components}
If an anticanonical curve $A$ consists of three components $A_1$, $A_2$
and $A_3$ these components must
have classes of the following three types
\begin{equation*}
    \begin{array}{lll}
         A_1 =  2L-E_1-\cdots-E_5, & A_2  =  L-E_i, & A_3=E_i,  \\
         A_1 = 2L-E_1-\cdots-E_5+E_i, & A_2 = L-E_i-E_j, & A_3 = E_i,   \\
         A_1 = L-E_i-E_j, & A_2 = L-E_k-E_l, & A_3 = L-E_m.
    \end{array}
\end{equation*}
Each of these cases gives a root system of type $A_3$ and $W(D_5)$ acts transitively on the set of anticanonical curves with three components.

\begin{Proposition}
 There are $W(D_5)$-equivariant isomorphisms
 \begin{equation*}
     \begin{array}{lcl}
     \DP_4^{\rm 3n} & \cong & (\ZZ/2\ZZ)^4 \setminus \bigsqcup_{W(D_5)/W(A_3)} T_{A_3} \\
      \DP_4^{\rm tn}  & \cong & (\ZZ/2\ZZ)^4 \setminus \bigsqcup_{W(D_5)/W(A_3)}\PP(V_{A_3}).
      \end{array}
 \end{equation*}
\end{Proposition}

\subsection{Anticanonical curves with four components}
An anticanonical curve $A$ consisting of four components $A_1$, $A_2$, $A_3$ and $A_4$ is
\begin{equation*}
    \begin{array}{llll}
         A_1 = 2L-E_1-\cdots-E_5,  A_2 = L-E_i-E_j, A_3 = E_i  A_4=E_j\text{ or }\\
         A_1 = L-E_i-E_j, A_2 = L-E_k-E_l, A_3 = L-E_k-E_m, A_4 = E_k.
    \end{array}
\end{equation*}
Both cases give a root system of type $A_2$ and $W(D_5)$ acts transitively on
the set of anticanonical curves with four components.
\begin{Proposition}
 There is a $W(D_5)$-equivariant isomorphism
 \begin{equation*}
     \begin{array}{lcl}
     \DP_4^{\rm 4n} & \cong & (\ZZ/2\ZZ)^4 \setminus \bigsqcup_{W(D_5)/W(A_2)} T_{A_2}
      \end{array}
 \end{equation*}
\end{Proposition}

\appendix
\section{Tables for cubic Del Pezzo surfaces}
\begin{scriptsize}
 \begin{table}[h]
 \begin{equation*}
 \begin{array}{|r|rrrrrrrrrrrrr|}
 \hline
 &&&&&&&&&&&&&\\[-1em]
 \, & \phi_{1}^0 & \phi_{1}^{36} & \phi_{6}^{25} & \phi_{6}^1 & \phi_{10}^{9} & \phi_{15}^{17} & \phi_{15}^{16} & \phi_{15}^{5} & \phi_{15}^{4} & \phi_{20}^{20} & \phi_{20}^2 & \phi_{20}^{10} & \phi_{24}^{12} \\
 &&&&&&&&&&&&&\\[-1em]
 \hline
 &&&&&&&&&&&&&\\[-1em]
\HH^0 & 1 & 0 & 0 & 0 & 0 & 0 & 0 & 0 & 0 & 0 & 0 & 0 & 0 \\
\HH^1 & 1 & 0 & 0 & 0 & 0 & 0 & 0 & 0 & 1 & 0 & 1 & 0 & 0 \\
\HH^2 & 0 & 0 & 0 & 0 & 0 & 0 & 0 & 1 & 2 & 0 & 2 & 0 & 0 \\
\HH^3 & 0 & 0 & 0 & 1 & 1 & 0 & 0 & 4 & 2 & 0 & 3 & 1 & 1 \\
\HH^4 & 0 & 0 & 0 & 3 & 6 & 3 & 1 & 8 & 4 & 1 & 8 & 5 & 4 \\
\HH^5 & 1 & 0 & 2 & 5 & 7 & 9 & 8 & 11 & 11 & 9 & 17 & 13 & 13 \\
\HH^6 & 2 & 1 & 2 & 3 & 2 & 8 & 14 & 8 & 15 & 13 & 16 & 15 & 19\\
\hline
 &&&&&&&&&&&&&\\[-1em]
\, & \phi_{24}^{6} & \phi_{30}^{15} & \phi_{30}^{3} & \phi_{60}^{11} & \phi_{60}^{5} & \phi_{60}^{8} & \phi_{64}^{13} & \phi_{64}^{4} & \phi_{80}^{7} & \phi_{81}^{6} & \phi_{81}^{10} & \phi_{90}^{8} & \, \\
 &&&&&&&&&&&&&\\[-1em]
\hline
 &&&&&&&&&&&&&\\[-1em]
\HH^0 & 0 & 0 & 0 & 0 & 0 & 0 & 0 & 0 & 0 & 0 & 0 & 0 & \, \\
\HH^1 & 0 & 0 & 0 & 0 & 0 & 0 & 0 & 0 & 0 & 0 & 0 & 0 & \, \\
\HH^2 & 0 & 0 & 1 & 0 & 1 & 1 & 0 & 2 & 0 & 2 & 0 & 0 & \, \\
\HH^3 & 2 & 0 & 6 & 2 & 6 & 4 & 1 & 9 & 7 & 9 & 3 & 7 & \, \\
\HH^4 & 8 & 7 & 17 & 14 & 23 & 15 & 13 & 26 & 31 & 27 & 20 & 31 & \, \\
\HH^5 & 18 & 17 & 23 & 37 & 45 & 40 & 38 & 48 & 55 & 56 & 52 & 61 & \,\\
\HH^6 & 21 & 11 & 12 & 32 & 34 & 44 & 36 & 39 & 37 & 54 & 53 & 49 & \, \\
 \hline
\end{array}
 \end{equation*}
 \caption{$\HH^i(\DP_3^{\rm n})$ as a
 representation of $W(E_6)$.}
 \label{nodecohtable}
 \end{table}
\end{scriptsize}

\begin{scriptsize}
 \begin{table}[h]
 \begin{equation*}
 \begin{array}{|r|rrrrrrrrrrrrr|}
 \hline
 &&&&&&&&&&&&&\\[-1em]
 \, & \phi_{1}^0 & \phi_{1}^{36} & \phi_{6}^{25} & \phi_{6}^1 & \phi_{10}^{9} & \phi_{15}^{17} & \phi_{15}^{16} & \phi_{15}^{5} & \phi_{15}^{4} & \phi_{20}^{20} & \phi_{20}^2 & \phi_{20}^{10} & \phi_{24}^{12} \\
 &&&&&&&&&&&&&\\[-1em]
 \hline
 &&&&&&&&&&&&&\\[-1em]
\HH^0 & 1 & 0 & 0 & 0 & 0 & 0 & 0 & 0 & 0 & 0 & 0 & 0 & 0 \\
\HH^1 & 0 & 0 & 0 & 0 & 0 & 0 & 0 & 0 & 1 & 0 & 1 & 0 & 0 \\
\HH^2 & 0 & 0 & 0 & 0 & 0 & 0 & 0 & 0 & 1 & 0 & 1 & 0 & 0 \\
\HH^3 & 0 & 0 & 0 & 1 & 1 & 0 & 0 & 3 & 1 & 0 & 2 & 1 & 1 \\
\HH^4 & 0 & 0 & 0 & 2 & 3 & 2 & 1 & 4 & 2 & 1 & 5 & 3 & 3 \\
\HH^5 & 0 & 0 & 1 & 2 & 1 & 3 & 4 & 3 & 5 & 4 & 6 & 6 & 6 \\
\hline
 &&&&&&&&&&&&&\\[-1em]
\, & \phi_{24}^{6} & \phi_{30}^{15} & \phi_{30}^{3} & \phi_{60}^{11} & \phi_{60}^{5} & \phi_{60}^{8} & \phi_{64}^{13} & \phi_{64}^{4} & \phi_{80}^{7} & \phi_{81}^{6} & \phi_{81}^{10} & \phi_{90}^{8} & \, \\
 &&&&&&&&&&&&&\\[-1em]
\hline
 &&&&&&&&&&&&&\\[-1em]
\HH^0 & 0 & 0 & 0 & 0  & 0  & 0  & 0  & 0  & 0  & 0  & 0  & 0  & \,\\
\HH^1 & 0 & 0 & 0 & 0  & 0  & 0  & 0  & 0  & 0  & 0  & 0  & 0  & \,\\
\HH^2 & 0 & 0 & 1 & 0  & 1  & 1  & 0  & 2  & 0  & 2  & 0  & 0  & \,\\
\HH^3 & 2 & 0 & 4 & 2  & 5  & 3  & 1  & 6  & 6  & 6  & 3  & 6  & \,\\
\HH^4 & 5 & 5 & 9 & 10 & 14 & 10 & 9  & 15 & 17 & 16 & 13 & 19 & \,\\
\HH^5 & 7 & 5 & 6 & 13 & 15 & 16 & 14 & 17 & 17 & 21 & 20 & 20 & \,\\
\hline
\end{array}
 \end{equation*}
 \caption{$\HH^i(\DP_3^{\rm c})$ as a
 representation of $W(E_6)$. See also \cite{fleischmannjaniszczak_hyp}.}
 \label{cuspcohtable}
 \end{table}
\end{scriptsize}

\begin{scriptsize}
 \begin{table}[h]
 \begin{equation*}
 \begin{array}{|r|rrrrrrrrrrrrr|}
 \hline
 &&&&&&&&&&&&&\\[-1em]
 \, & \phi_{1}^0 & \phi_{1}^{36} & \phi_{6}^{25} & \phi_{6}^1 & \phi_{10}^{9} & \phi_{15}^{17} & \phi_{15}^{16} & \phi_{15}^{5} & \phi_{15}^{4} & \phi_{20}^{20} & \phi_{20}^2 & \phi_{20}^{10} & \phi_{24}^{12}\\
 &&&&&&&&&&&&&\\[-1em]
 \hline
 &&&&&&&&&&&&&\\[-1em]
\HH^0 & 1 & 0 & 0 & 1 & 0 & 0 & 0 & 0 & 0 & 0 & 1 & 0 & 0 \\
\HH^1 & 1 & 0 & 0 & 1 & 0 & 0 & 0 & 0 & 2 & 0 & 3 & 0 & 0 \\
\HH^2 & 0 & 0 & 0 & 0 & 0 & 0 & 1 & 2 & 5 & 0 & 5 & 2 & 0 \\
\HH^3 & 0 & 0 & 0 & 2 & 2 & 2 & 3 & 8 & 7 & 2 & 9 & 8 & 1 \\
\HH^4 & 0 & 0 & 3 & 6 & 9 & 8 & 6 & 13 & 8 & 9 & 14 & 12 & 1 \\
\HH^5 & 0 & 0 & 4 & 5 & 9 & 8 & 4 & 9 & 4 & 8 & 9 & 8 & 0  \\
\hline
 &&&&&&&&&&&&&\\[-1em]
\,  & \phi_{24}^{6} & \phi_{30}^{15} & \phi_{30}^{3} & \phi_{60}^{11} & \phi_{60}^{5} & \phi_{60}^{8} & \phi_{64}^{13} & \phi_{64}^{4} & \phi_{80}^{7} & \phi_{81}^{6} & \phi_{81}^{10} & \phi_{90}^{8} & \, \\
 &&&&&&&&&&&&&\\[-1em]
\hline
 &&&&&&&&&&&&&\\[-1em]
\HH^0 & 0 & 0 & 0 & 0 & 0 & 0 & 0 & 0 & 0 & 0 & 0 & 0 & \, \\
\HH^1 & 1 & 0 & 1 & 0 & 2 & 1 & 0 & 2 & 0 & 1 & 0 & 0 & \, \\
\HH^2 & 2 & 0 & 4 & 1 & 6 & 3 & 0 & 8 & 1 & 8 & 2 & 3 & \, \\
\HH^3 & 2 & 4 & 12 & 8 & 14 & 6 & 7 & 20& 10 & 19 & 9 & 16 & \, \\
\HH^4 & 1 & 19 & 26 & 26 & 30 & 11 & 25 & 34 & 33 & 28 & 22 & 36 & \, \\
\HH^5 & 0 & 21 & 23 & 26 & 27 & 9 & 25 & 27 & 36 & 22 & 21 & 35 & \, \\
\hline
\end{array}
 \end{equation*}
 \caption{$\HH^i(\DP_3^{\rm 2n})$ as
 a representation of $W(E_6)$.}
 \label{clcohtable}
 \end{table}
\end{scriptsize}

\begin{scriptsize}
 \begin{table}[h]
 \begin{equation*}
 \begin{array}{|r|rrrrrrrrrrrrr|}
 \hline
 &&&&&&&&&&&&&\\[-1em]
 \, & \phi_{1}^0 & \phi_{1}^{36} & \phi_{6}^{25} & \phi_{6}^1 & \phi_{10}^{9} & \phi_{15}^{17} & \phi_{15}^{16} & \phi_{15}^{5} & \phi_{15}^{4} & \phi_{20}^{20} & \phi_{20}^2 & \phi_{20}^{10} & \phi_{24}^{12} \\
 &&&&&&&&&&&&&\\[-1em]
 \hline
 &&&&&&&&&&&&&\\[-1em]
\HH^0 & 1 & 0 & 0 & 1 & 0 & 0 & 0 & 0 & 0 & 0 & 1 & 0 & 0 \\
\HH^1 & 0 & 0 & 0 & 0 & 0 & 0 & 0 & 0 & 2 & 0 & 2 & 0 & 0 \\
\HH^2 & 0 & 0 & 0 & 0 & 0 & 0 & 1 & 1 & 3 & 0 & 3 & 1 & 1 \\
\HH^3 & 0 & 0 & 0 & 2 & 2 & 2 & 2 & 5 & 3 & 2 & 5 & 5 & 5 \\
\HH^4 & 0 & 0 & 1 & 2 & 3 & 3 & 2 & 4 & 2 & 3 & 4 & 4 & 4 \\
\hline
 &&&&&&&&&&&&&\\[-1em]
\, & \phi_{24}^{6} & \phi_{30}^{15} & \phi_{30}^{3} & \phi_{60}^{11} & \phi_{60}^{5} & \phi_{60}^{8} & \phi_{64}^{13} & \phi_{64}^{4} & \phi_{80}^{7} & \phi_{81}^{6} & \phi_{81}^{10} & \phi_{90}^{8} & \, \\
 &&&&&&&&&&&&&\\[-1em]
\hline
 &&&&&&&&&&&&&\\[-1em]
\HH^0 & 0 & 0 & 0 & 0  & 0  & 0  & 0  & 0  & 0  & 0  & 0  & 0  & \, \\
\HH^1 & 1 & 0 & 1 & 0  & 2  & 1  & 0  & 2  & 0  & 1  & 0  & 0  & \, \\
\HH^2 & 3 & 0 & 3 & 2  & 6  & 5  & 1  & 7  & 4  & 9  & 5  & 5  & \, \\
\HH^3 & 6 & 3 & 7 & 10 & 13 & 13 & 10 & 16 & 16 & 18 & 15 & 19 & \, \\
\HH^4 & 4 & 7 & 9 & 13 & 14 & 11 & 13 & 15 & 20 & 17 & 16 & 21 & \, \\
\hline
\end{array}
 \end{equation*}
 \caption{$\HH^i(\DP_3^{\rm tn})$ as
 a representation of $W(E_6)$.}
 \label{ctlcohtable}
 \end{table}
\end{scriptsize}


\begin{scriptsize}
 \begin{table}[h]
 \begin{equation*}
 \begin{array}{|r|rrrrrrrrrrrrr|}
 \hline
 &&&&&&&&&&&&&\\[-1em]
 \, & \phi_{1}^0 & \phi_{1}^{36} & \phi_{6}^{25} & \phi_{6}^1 & \phi_{10}^{9} & \phi_{15}^{17} & \phi_{15}^{16} & \phi_{15}^{5} & \phi_{15}^{4} & \phi_{20}^{20} & \phi_{20}^2 & \phi_{20}^{10} & \phi_{24}^{12} \\
&&&&&&&&&&&&&\\[-1em]
\hline
&&&&&&&&&&&&&\\[-1em]
\HH^0 & 1 & 0 & 0 & 0 & 0 & 0 & 0 & 0 & 0 & 0 & 1 & 0 & 0 \\
\HH^1 & 2 & 0 & 0 & 1 & 0 & 0 & 1 & 0 & 3 & 0 & 5 & 0 & 0 \\
\HH^2 & 2 & 0 & 0 & 3 & 0 & 1 & 7 & 1 & 12 & 4 & 14 & 3 & 5 \\
\HH^3 & 3 & 2 & 2 & 5 & 1 & 9 & 21 & 8 & 26 & 19 & 28 & 16 & 24 \\
\HH^4 & 3 & 3 & 4 & 5 & 3 & 15 & 26 & 14 & 28 & 27 & 30 & 25 & 35 \\
\hline
&&&&&&&&&&&&&\\[-1em]
\, & \phi_{24}^{6} & \phi_{30}^{15} & \phi_{30}^{3} & \phi_{60}^{11} & \phi_{60}^{5} & \phi_{60}^{8} & \phi_{64}^{13} & \phi_{64}^{4} & \phi_{80}^{7} & \phi_{81}^{6} & \phi_{81}^{10} & \phi_{90}^{8} & \, \\
&&&&&&&&&&&&&\\[-1em]
\hline
&&&&&&&&&&&&&\\[-1em]
\HH^0 & 1 & 0 & 0 & 0 & 0 & 0 & 0 & 0 & 0 & 0 & 0 & 0 & \, \\
\HH^1 & 4 & 0 & 0 & 2 & 0 & 3 & 0 & 3 & 0 & 2 & 2 & 0 & \, \\
\HH^2 & 12 & 0 & 3 & 15 & 7 & 21 & 7 & 18 & 6 & 20 & 17 & 8  & \, \\
\HH^3 & 32 & 7 & 13 & 49 & 37 & 65 & 40 & 53 & 36 & 70 & 67 & 51 & \, \\
\HH^4 & 39 & 17 & 20 & 64 & 58 & 83 & 65 & 70 & 62 & 97 & 97 & 85 & \, \\
\hline
\end{array}
\end{equation*}
\caption{$\HH^i\left( \bigsqcup_{(f_1,f_2,f_3)} (\ZZ/2\ZZ)\backslash T_{F_4}(f_1) \right)$ as a representation of $W(E_6)$.}
\label{tlcohtable}
\end{table}
\end{scriptsize}


\begin{scriptsize}
 \begin{table}[h]
 \begin{equation*}
 \begin{array}{|r|rrrrrrrrrrrrr|}
 \hline
 &&&&&&&&&&&&&\\[-1em]
 \, & \phi_{1}^0 & \phi_{1}^{36} & \phi_{6}^{25} & \phi_{6}^1 & \phi_{10}^{9} & \phi_{15}^{17} & \phi_{15}^{16} & \phi_{15}^{5} & \phi_{15}^{4} & \phi_{20}^{20} & \phi_{20}^2 & \phi_{20}^{10} & \phi_{24}^{12} \\
 &&&&&&&&&&&&&\\[-1em]
 \hline
 &&&&&&&&&&&&&\\[-1em]
\HH^0 & 1 & 0 & 0 & 0 & 0 & 0 & 0 & 0 & 0  & 0 & 1  & 0 & 0  \\
\HH^1 & 1 & 0 & 0 & 1 & 0 & 0 & 1 & 0 & 3  & 0 & 4  & 0 & 0  \\
\HH^2 & 1 & 0 & 0 & 2 & 0 & 1 & 5 & 0 & 8  & 3 & 9  & 3 & 3  \\
\HH^3 & 1 & 1 & 1 & 2 & 1 & 5 & 8 & 5 & 10 & 8 & 11 & 7 & 11 \\
\hline
 &&&&&&&&&&&&&\\[-1em]
\, & \phi_{24}^{6} & \phi_{30}^{15} & \phi_{30}^{3} & \phi_{60}^{11} & \phi_{60}^{5} & \phi_{60}^{8} & \phi_{64}^{13} & \phi_{64}^{4} & \phi_{80}^{7} & \phi_{81}^{6} & \phi_{81}^{10} & \phi_{90}^{8} & \, \\
 &&&&&&&&&&&&&\\[-1em]
\hline
 &&&&&&&&&&&&&\\[-1em]
\HH^0 & 1  & 0 & 0 & 0  & 0  & 0  & 0  & 0  & 0  & 0  & 0  & 0  & \, \\
\HH^1 & 3  & 0 & 0 & 0  & 2  & 3  & 0  & 3  & 0  & 2  & 2  & 0  & \, \\
\HH^2 & 8  & 0 & 3 & 5  & 12 & 16 & 6  & 14 & 6  & 15 & 13 & 7  & \, \\
\HH^3 & 14 & 4 & 7 & 18 & 23 & 27 & 18 & 23 & 18 & 31 & 30 & 26 & \, \\
\hline
\end{array}
 \end{equation*}
 \caption{$\HH^i(\DP_3^{\rm tp})$ as
 a representation of $W(E_6)$.}
 \label{Ecohtable}
 \end{table}
\end{scriptsize}

\clearpage
\section{Tables for quartic Del Pezzo surfaces}

\begin{scriptsize}
\begin{table}[h]
\begin{equation*}
    \begin{array}{|l|rrrrrrr|}
    \hline
 &&&&&&&\\[-1em]
    \, & s_5 & s_{1^5} & s_{4,1} & s_{2,1^3} & s_{3,2} & s_{2^2,1} & s_{3,1^2} \\
 &&&&&&&\\[-1em]
    \hline
 &&&&&&&\\[-1em]
\HH^i(\DP_4^{\rm n}) & 1 & 0 & 0 & 0 & 0 & 0 & 0 \\
 & 1 & 0 & 0 & 1 & 0 & 1 & 0 \\
 & 0 & 0 & 0 & 2 & 1 & 2 & 2 \\
 & 0 & 0 & 1 & 3 & 2 & 3 & 4 \\
 & 1 & 0 & 3 & 5 & 4 & 5 & 6 \\
 & 2 & 1 & 4 & 5 & 6 & 6 & 6 \\
    \hline
    \hline
 &&&&&&&\\[-1em]
\HH^i(\DP_4^{\rm c}) & 1 & 0 & 0 & 0 & 0 & 0 & 0 \\
 & 0 & 0 & 0 & 1 & 1 & 0 & 0 \\
 & 0 & 0 & 0 & 1 & 1 & 1 & 2 \\
 & 0 & 0 & 1 & 1 & 1 & 1 & 1 \\
    \hline
    \hline
 &&&&&&&\\[-1em]
\HH^i(\DP_4^{\rm 2n, A_4}) & 1 & 0 & 0 & 0 & 0 & 0 & 0 \\
 & 1 & 0 & 0 & 2 & 1 & 0 & 0 \\
 & 1 & 0 & 0 & 4 & 4 & 2 & 4 \\
 & 1 & 0 & 4 & 6 & 7 & 6 & 8 \\
 & 1 & 1 & 4 & 4 & 5 & 5 & 6 \\
    \hline
    \hline
 &&&&&&&\\[-1em]
\HH^i(\DP_4^{\rm tn, A_4}) & 1 & 0 & 0 & 0 & 0 & 0 & 0 \\
 & 0 & 0 & 0 & 1 & 1 & 0 & 0 \\
 & 0 & 0 & 0 & 1 & 1 & 1 & 2 \\
 & 0 & 0 & 1 & 1 & 1 & 1 & 1  \\
    \hline
    \hline
 &&&&&&&\\[-1em]
\HH^i(\DP_4^{\rm 2n, D_4}) & 1 & 0 & 0 & 1 & 0 & 0 & 0 \\
 & 1 & 0 & 0 & 2 & 2 & 1 & 1 \\
 & 0 & 0 & 1 & 3 & 4 & 2 & 4 \\
 & 1 & 0 & 3 & 5 & 5 & 4 & 7 \\
 & 2 & 1 & 4 & 5 & 6 & 6 & 6 \\
    \hline
    \hline
 &&&&&&&\\[-1em]
\HH^i(\DP_4^{\rm tn, D_4}) & 1 & 0 & 0 & 1 & 0 & 0 & 0 \\
 & 0 & 0 & 0 & 1 & 2 & 1 & 1 \\
 & 0 & 0 & 1 & 2 & 2 & 1 & 3 \\
 & 1 & 0 & 1 & 2 & 2 & 2 & 2 \\
    \hline
    \hline
 &&&&&&&\\[-1em]
\HH^i(\DP_4^{\rm 3n}) & 1 & 0 & 0 & 1 & 0 & 0 & 0 \\
 & 1 & 0 & 0 & 3 & 3 & 1 & 2 \\
 & 1 & 0 & 3 & 5 & 6 & 5 & 7 \\
 & 1 & 1 & 4 & 4 & 5 & 5 & 6 \\
    \hline
    \hline
 &&&&&&&\\[-1em]
\HH^i(\DP_4^{\rm tp}) & 1 & 0 & 0 & 1 & 0 & 0 & 0 \\
 & 0 & 0 & 0 & 1 & 2 & 1 & 1 \\
 & 0 & 0 & 1 & 1 & 1 & 1 & 2 \\
    \hline
    \hline
 &&&&&&&\\[-1em]
\HH^i(\DP_4^{\rm 4n}) & 1 & 0 & 0 & 2 & 0 & 1 & 1 \\
 & 1 & 0 & 2 & 4 & 4 & 5 & 5 \\
 & 1 & 1 & 4 & 4 & 5 & 5 & 6 \\
    \hline
\end{array}
\end{equation*}
\caption{The cohomology groups $\HH^i(\DP_4^{\rm *},\QQ)$ as representations of
$S_5=W(D_5)/(\ZZ/2\ZZ)^4$. For every singularity type listed in the first column, the $i$-th row's values are the multiplicities of the representation $s_\alpha$ from the first row that appear in $\HH^i$.}
 \label{dp4table}
\end{table}
\end{scriptsize}

\bibliographystyle{alpha}

\renewcommand{\bibname}{References}


\newcommand{\etalchar}[1]{$^{#1}$}

\end{document}